\documentclass{article}
\usepackage{graphicx} 
\usepackage{amsfonts}
\usepackage{amsmath,amssymb}
\usepackage{amsthm}
\usepackage{bm}
\usepackage{xcolor}
\usepackage{authblk}
\usepackage{url}
\usepackage{wrapfig}
\usepackage{booktabs}
\usepackage{multirow}
\usepackage{caption}
\usepackage{subcaption}

\allowdisplaybreaks
\newcommand{\lJump}{[\![}
\newcommand{\rJump}{]\!]}
\newtheorem{theorem}{Theorem}

\newtheorem{remark}{Remark}
\newtheorem{lemma}{Lemma}

\newcommand{\tn}{{\textnormal{tan}}}
\newcommand\numberthis{\addtocounter{equation}{1}\tag{\theequation}}
\allowdisplaybreaks

\let\originallesssim\lesssim
\let\originalgtrsim\gtrsim

\DeclareRobustCommand{\lesssim}{%
  \mathrel{\mathpalette\lowersim\originallesssim}%
}
\DeclareRobustCommand{\gtrsim}{%
  \mathrel{\mathpalette\lowersim\originalgtrsim}%
}

\makeatletter
\newcommand{\lowersim}[2]{%
  \sbox\z@{$#1<$}%
  \raisebox{-\dimexpr\height-\ht\z@}{$\m@th#1#2$}%
}

\begin{document}

\title{ Analysis of a Surface Crouzeix-Raviart Element for the Stokes Problem }
\author{Manisha Chowdhury \thanks{manisha.chowdhury@ovgu.de}, Carolin Mehlmann \thanks{carolin.mehlmann@ovgu.de}}
\affil{\footnotesize Institute of Analysis and Numerics,
    Otto-von-Guericke University Magdeburg, Universitätsplatz 2,
  39106 Magdeburg, Germany}
\date{}

\maketitle

\begin{abstract}
Recently, increasing attention has been paid to finite element discretizations of vector-valued flow problems posed on curved surfaces. In this work, we study a surface Stokes system defined on a two-dimensional manifold embedded in three-dimensional space. The surface velocity field is approximated using a nonconforming Crouzeix-Raviart finite element on a polyhedral approximation of the surface, while the pressure is discretized by piecewise constant
functions. The governing equations involve the symmetric strain-rate tensor, whose approximation with Crouzeix-Raviart finite elements leads to spurious oscillations in the velocity field because the discrete Korn inequality fails on the nonconforming space. We therefore stabilize the momentum equation by adding an edge-jump penalty term, which restores the coervity of the bilinear term as the discrete Korn inequality holds for this stabilized version. Additionally, we establish that this finite element pair of  velocity and pressure spaces satisfies the discrete inf-sup compatibility condition. Furthermore, we derive optimal a-priori error estimates in both the energy norm and the  $L^2$-norm. These theoretical results are corroborated by numerical
experiments.
\end{abstract}

\section{Introduction}
Surface PDEs have attracted significant attention in recent years due to their wide range of applications in modeling emulsions, foams and biological membranes \textcolor{black}{\cite{arroyo2009, brenner2013, rahimi2013, rangamani2013, scriven1960, slattery2007}}, biophysics \cite{torres2019}, and large-scale climate systems \cite{COMBLEN2009, Mehlmann2021, MehlmannGutjahr2022}. Modern climate models require robust and numerically affordable numerical discretizations for flows on the surface of the Earth. For example, in climate modeling context, surface Stokes equations arise as a simplified model describing incompressible viscous flows constrained to evolve tangentially along the Earth's spherical geometry. Low-order nonconforming finite element methods offer a favorable balance between computational efficiency and numerical stability in the approximation of such surface PDEs \textcolor{black}{used in climate models \cite{Mehlmann2021, MehlmannGutjahr2022}.} \vspace{1mm}\\
Motivated by the numerical discretizations employed in approximation of the sea-ice dynamics in the ICON climate model \cite{JUNGCLAUS_2020}, we analyze a surface Crouzeix-Raviart finite element method for the surface Stokes equations in this paper. Earlier, Guo in \cite{guo2018} proposed and studied a scalar surface finite element method for the Laplace-Beltrami equation and later, Mehlmann in \cite{mehlmann2023} developed and analyzed a vector-valued surface Crouzeix-Raviart element for the vector-valued Laplacian problem. 
A major challenge in the construction of finite element methods for surface flow problems is the enforcement of tangentiality of the discrete velocity field with respect to the surface. In \cite{hardering2025, jankuhn2021, jankuhn2019, reusken2025}, $H^1$-conforming spaces were designed to approximate the discrete solutions, while the tangential constraint was imposed weakly through stabilization techniques such as using penalty terms or Lagrange multipliers. Besides, the authors in \cite{bonito2020, lederer2020a} developed tangential penalty free $H(\text{div}_{\Gamma})$-conforming finite element spaces for the surface Stokes problem. More recently, \cite{demlow2024} proposed a tangential penalty-free finite element method, while \cite{demlow2026} developed Taylor-Hood method without penalization for the numerical approximation of the surface Stokes equations.\vspace{1mm}\\
In this paper, we employ a low-order, simple to implement, computationally affordable tangential Crouzeix-Raviart finite element space \cite{mehlmann2023, mehlmann2025, Mehlmannetal2021} to approximate the vector-valued velocity, while the pressure variable is discretized by piecewise constant \textcolor{black}{functions.} First, the surface is approximated using polyhedrons and in each of its faces we establish  a local basis and express the Crouzeix-Raviart approximation in terms of the conormal vector $\boldsymbol{n}_E$ and the tangent vector $\boldsymbol{\tau}_E$ given in the edge midpoint \textcolor{black}{of a cell}. While $\boldsymbol{\tau}_E$ is continuous, $\boldsymbol{n}_E$ is discontinuous along an edge, see the right panel of Figure \ref{fig1}. By this construction, the elements of the Crouzeix-Raviart space are  tangential to the discrete surface, but $H^1(\Gamma_h)$ non-conforming. \vspace{1mm}\\
The novelty of this work lies in the rigorous analysis of this nonconforming stabilized formulation for the surface Stokes problem. The governing equations involve the symmetric strain-rate tensor, whose approximation with Crouzeix-Raviart finite elements leads to spurious oscillations in the velocity field due to the presence of a non-trivial kernel \cite{acosta2011}. To overcome \textcolor{black}{stability issues related to the symmetric tensor, $ \frac{1}{2}(\nabla_{\Gamma} \boldsymbol{u} +  \nabla_{\Gamma}^T \boldsymbol{u})$,} we consider a stabilized formulation of the momentum equation \cite{hansbo2003, mehlmann2025}. 
We prove the discrete inf-sup condition with respect to the energy norm, which includes the jump terms along the edges. To establish the compatibility condition related to the jump terms of the velocity, we introduce a suitable interpolation operator and derive the required interpolation estimates with respect to appropriate norms. Furthermore, in the energy-error analysis, additional nonconformity contributions arise due to the stabilization of the equation. The latter requires a careful treatment of the jump terms associated with discontinuities in the conormal vectors. In addition, the derivation of the $L^2$-error estimate constitutes one of the main analytical challenges of the present work. \textcolor{black}{Compared with the error analysis for scalar and vector-valued surface Laplace problems, the present setting involves additional analytical challenges arising from the pressure variable, the symmetric strain-rate operator, and the nonconforming velocity approximation. Consequently, the $L^2$-error estimate analysis requires a careful treatment of geometric approximation errors, velocity-pressure coupling terms, primal-dual consistency errors and nonconformity contributions.} 
These contributions are bounded through a combination of geometric approximation estimates, norm equivalence results, operator difference estimates, interpolation properties of the tangential Crouzeix-Raviart operator, and \textcolor{black}{the energy-error bound, established in this paper.} The proposed formulation provides a computationally efficient and easy-to-implement alternative to existing conforming and $H(\text{div}_{\Gamma})$-based approaches for the Stokes problem, while ensuring that the velocity vectors are tangential to the discrete surfaces. \vspace{1mm}\\
The organization of the paper is as follows: Section $\ref{sec2}$ presents the fundamental notions and Section $\ref{sec3}$ introduces the governing equations and its corresponding weak formulation. Section $\ref{sec4}$ presents the surface approximation followed by the construction of the finite \textcolor{black}{element} space and the discrete formulation, and contains the derivation of the discrete inf-sup compatibility condition. In Section $\ref{sec5}$, we derive the interpolation, energy-error, and $L^2$-error estimates. The theoretical results established in Section \ref{sec5} are numerically verified in Section \ref{sec6}.

\section{Preliminaries}\label{sec2}
This section introduces the fundamental concepts and notations used throughout the paper. Let $\Gamma \subset \mathbb{R}^3$ denote an oriented, connected, bounded, $C^\infty$-smooth surface with $\partial \Gamma=\emptyset$. At a point $\boldsymbol{x} \in \Gamma$ the outward unit normal vector $\boldsymbol{n}(\boldsymbol{x})$ is defined as $$\boldsymbol{n}(\boldsymbol{x})=\nabla d(\boldsymbol{x}),$$ where $ \nabla$ denotes the standard gradient operator in $\mathbb{R}^3$ and $d(\boldsymbol{x})$ is the signed distance function of $\Gamma$. Let $U_{\delta}$ be a tubular neighborhood of $\Gamma$ with width $\delta > 0$. For sufficiently small $\delta$, there exists a unique closest-point projection $p: U_{\delta} \rightarrow \Gamma$ satisfying $$p(\boldsymbol{x})= \boldsymbol{x}-d(\boldsymbol{x}) \boldsymbol{n}(p(\boldsymbol{x})).$$ 
Given a scalar function $\psi: \Gamma \rightarrow \mathbb{R}$, we define its extension $\psi^e : U_{\delta} \rightarrow \mathbb{R}$ by $$\psi^e(\boldsymbol{x}) = \psi(p(\boldsymbol{x})), \hspace{4mm} \forall \boldsymbol{x} \in U_{\delta}.$$ This construction extends $\psi$ smoothly along the normal direction to $\Gamma$. Analogously, we can lift a function $\phi: U_{\delta} \rightarrow \mathbb{R}$ using the lifting function $\phi^l: \Gamma \rightarrow \mathbb{R}$ defined as follows $$\phi^l(\boldsymbol{x})= \phi(\boldsymbol{\eta}(\boldsymbol{x})), \hspace{4mm} \forall \boldsymbol{x} \in \Gamma,$$ where $\boldsymbol{\eta}: \Gamma \rightarrow U_{\delta}$ is uniquely determined by $$\boldsymbol{x}=p(\boldsymbol{\eta})=\boldsymbol{\eta}-d(\boldsymbol{\eta}) \boldsymbol{n}(\boldsymbol{x}).$$ For vector-valued functions, we extend this lifting and extension framework component-wise. These two operators provide a convenient framework for transferring functions between the exact surface $\Gamma$ and its discrete approximation $\Gamma_h$, defined in Section 3. \vspace{2mm}\\
To introduce surface differential operators on $\Gamma$, we define the orthogonal projection operator $$\textbf{P} := \textbf{I}-\boldsymbol{n} \boldsymbol{n}^T,$$ which maps vectors in $\mathbb{R}^3$ onto the tangent plane of $\Gamma$. Here $\textbf{I}$ denotes the identity matrix in $\mathbb{R}^3 \times \mathbb{R}^3$. The covariant derivative of a vector field $\boldsymbol{w}$ is defined as $$\nabla_{\Gamma} \boldsymbol{w}: = \textbf{P} \nabla \boldsymbol{w}^e \textbf{P},$$ where $\boldsymbol{w}^e:U_{\delta} \rightarrow \mathbb{R}^3$ is the extension to $\boldsymbol{w}: \Gamma \rightarrow \mathbb{R}^3$. \vspace{1mm}\\
The surface divergence operator acting on the vector field $\boldsymbol{w}$ is defined by $$\text{div}_{\Gamma} \boldsymbol{w} := \text{tr}(\nabla_{\Gamma} \boldsymbol{w})= \text{tr}(\textbf{P} \nabla \boldsymbol{w}^e \textbf{P}).$$ For a matrix valued function $A: \Gamma \rightarrow \mathbb{R}^3 \times \mathbb{R}^3$ we define the divergence as follows: 
$$\text{div}_{\Gamma}(A): = \left( \text{div}_{\Gamma}(\textbf{e}_1^TA), \text{div}_{\Gamma}(\textbf{e}_2^TA), \text{div}_{\Gamma}(\textbf{e}_3^TA) \right),$$ where $\{\textbf{e}_1, \textbf{e}_2, \textbf{e}_3 \}$ are the standard basis vectors of $\mathbb{R}^3$. \vspace{2mm}\\
For $\Omega \subset \Gamma$ the $L^2$-inner product is denoted by $(\cdot, \cdot)_{\Omega}$ and the standard $L^2$-norm and $L^{\infty}$-norm are denoted by $\| \cdot \|_{L^2(\Omega)}$ and $\| \cdot \|_{L^{\infty}(\Omega)}$, respectively. $L^2_0(\Omega)$ is the space of square-integrable functions on $\Omega$ with zero mean value. For scalar-valued functions, the standard Sobolev spaces are given by $H^m(\Omega)$, for $m=1,2$. Let $\boldsymbol{H}^m(\Omega)$ denote the Sobolev space $(H^m(\Omega))^3$ for vector-valued functions on the surface with the norm and semi-norm as $$\| \boldsymbol{w} \|_{\boldsymbol{H}^m(\Omega)}^2: = \sum_{i=0}^m \| (\nabla \textbf{P})^i \boldsymbol{w}^e  \|^2_{\boldsymbol{L}^2(\Omega)} \text{ and } |\boldsymbol{w}|_{\boldsymbol{H}^m(\Omega)}^2: =\| (\nabla \textbf{P})^m \boldsymbol{w}^e  \|^2_{\boldsymbol{L}^2(\Omega)},$$ respectively. Here, analogously we denote the space $(L^2(\Omega))^3$ by $\boldsymbol{L}^2(\Omega)$ for the vector-valued functions. Now, the corresponding spaces for tangential vector-valued functions are given as 
$$\boldsymbol{L}^2_{\text{tan}}(\Gamma): = \{\boldsymbol{w} \in \boldsymbol{L}^2(\Gamma) \mid \boldsymbol{w} \cdot \boldsymbol{n} = 0  \},$$
$$\boldsymbol{H}^m_{\text{tan}}(\Gamma): = \{\boldsymbol{w} \in \boldsymbol{H}^m(\Gamma) \mid \boldsymbol{w} \cdot \boldsymbol{n} = 0  \}.$$ 
$\boldsymbol{L}^2_{\text{tan}}(\Gamma)$ is equipped with the standard $\boldsymbol{L}^2$-norm over $\Gamma$ and the norm corresponding to $\boldsymbol{H}^m_{\text{tan}}(\Gamma)$ is $$\| \boldsymbol{w} \|_{\boldsymbol{H}^m_{\text{tan}}(\Omega)}:= \sum\limits_{i=0}^m \| (\nabla_{\Gamma})^i \boldsymbol{w} \|^2_{\boldsymbol{L}^2(\Omega)}.$$
From Lemma 2.3 in \cite{HansboLarson2020}, it follows that, for any $\boldsymbol{w} \in \boldsymbol{H}^m_{\text{tan}}(\Gamma)$, the norms $\| \boldsymbol{w} \|_{\boldsymbol{H}^m(\Gamma)}$ and $\| \boldsymbol{w} \|_{\boldsymbol{H}^m_{\text{tan}}(\Gamma)}$ are equivalent.

\section{Continuous Problem}\label{sec3}
The tangential surface Stokes equations \cite{bonito2020, demlow2024, lederer2020a} are to find a tangential velocity vector field $\boldsymbol{u}: \Gamma \rightarrow \mathbb{R}^3$ and a scalar pressure field $p: \Gamma \rightarrow \mathbb{R}$ with $\int_{\Gamma} p \hspace{1mm} ds= 0$, such that
\begin{equation}
\begin{split}
    \label{eq1}
    -  \textbf{P} \text{div}_{\Gamma} (E_{\Gamma} (\boldsymbol{u})) + \boldsymbol{u} - \nabla_{\Gamma} p & = \boldsymbol{f}, \hspace{4mm} \text{on} \hspace{2mm} \Gamma, \\
    \text{div}_{\Gamma} \boldsymbol{u} & = 0, \hspace{4.5mm} \text{on} \hspace{2mm} \Gamma,
\end{split}
\end{equation}
where $E_{\Gamma}(\boldsymbol{u})= \frac{1}{2}\left( \nabla_{\Gamma} \boldsymbol{u} + ( \nabla_{\Gamma} \boldsymbol{u})^T \right)$ is the tangential strain rate tensor. Since $\Gamma$ is a closed surface, no boundary conditions are imposed. The zeroth-order term $\boldsymbol{u}$ is introduced to avoid technical difficulties associated with the non-trivial kernel of the surface strain operator on closed surfaces.\vspace{2mm}\\
The variational formulation is to find a tangential velocity field $\boldsymbol{u} \in \boldsymbol{H}^1_{\text{tan}}(\Gamma)$ and a pressure field $p \in L_0^2(\Gamma)$ such that
\begin{equation}
    \begin{split}
    \label{eq2}
         a(\boldsymbol{u}, \boldsymbol{v}) + b(\boldsymbol{v},p) & = l(\boldsymbol{v}) \hspace{4mm} \forall \boldsymbol{v} \in \boldsymbol{H}^1_{\text{tan}}(\Gamma), \\
    b(\boldsymbol{u},q) & = 0 \hspace{8.9mm} \forall q \in L_0^2(\Gamma),
    \end{split}
\end{equation}
where the bilinear and linear functional forms are 
\begin{align}
    a(\boldsymbol{u}, \boldsymbol{v}) & =  \int_{\Gamma} E_{\Gamma} (\boldsymbol{u}) : E_{\Gamma} (\boldsymbol{v}) \, ds + \int_{\Gamma} \boldsymbol{u} \cdot \boldsymbol{v} \, ds= (E_{\Gamma} (\boldsymbol{u}), E_{\Gamma} (\boldsymbol{v}))_{\Gamma} + (\boldsymbol{u}, \boldsymbol{v})_{\Gamma}, \\
    b(\boldsymbol{u}, q) & = \int_{\Gamma} (\text{div}_{\Gamma} \boldsymbol{u}) q \, ds = (q, \text{div}_{\Gamma} \boldsymbol{u})_{\Gamma}, \\
    l(\boldsymbol{v}) & = (\boldsymbol{f}, \boldsymbol{v})_{\Gamma}.
\end{align}
It is straightforward to verify that the terms $a(\cdot, \cdot)$, $b(\cdot,\cdot)$ and $l(\cdot)$ given in equations (3)-(5) are continuous with respect to the standard norms on the corresponding spaces. For $C^2$-smooth, compact and closed surface $\Gamma$, the coercivity of $a(\cdot, \cdot)$:
 \begin{equation}
    \label{eq6}
      \| \boldsymbol{u} \|_{\boldsymbol{L}^2(\Gamma)} + \| E_{\Gamma} (\boldsymbol{u}) \|_{\boldsymbol{L}^2(\Gamma)} \geq C_1 \| \boldsymbol{u} \|_{\boldsymbol{H}^1(\Gamma)} \hspace{4mm} \forall \boldsymbol{u} \in \boldsymbol{H}^1_{\tan} (\Gamma),
    \end{equation}
and the continuous inf-sup condition:
 \begin{equation}
    \label{eq7}
        \underset{\boldsymbol{v} \in \boldsymbol{H}^1_{\tan}(\Gamma)}{sup} \frac{b(\boldsymbol{v}, q)}{\| \boldsymbol{v} \|_{\boldsymbol{H}^1_{\tan}(\Gamma)}} \geq \beta \| q \|_{L^2(\Gamma)} \hspace{4mm} \forall q \in L^2_{0} (\Gamma),
    \end{equation}   
have been established by Jankuhn et. al. in \cite{jankuhn2018} (see Lemma 4.1 and Lemma 4.2), where $C_1$ and $\beta$ are positive constants. They have also shown the well-posedness of the variational formulation (\ref{eq2}) in Theorem 4.3 in \cite{jankuhn2018}. Further Olshanskii et. al. in \cite{olshanskii2018} (see Lemma 2.1) established the following regularity estimate (\ref{eq8}), which play central role in the derivation of the \textcolor{black}{a priori} error estimations in Section \ref{sec5}.
    \begin{lemma}\label{lem3}{} 
     Assume $\Gamma$ is $C^2$ smooth, compact and closed. The weak problem (\ref{eq2}) is well-posed and let $(\boldsymbol{u},p)$ be the unique solution. If $\boldsymbol{f} \in \boldsymbol{L}^2_{\tan}(\Gamma)$, then $\boldsymbol{u} \in \boldsymbol{H}_{\tan}^2(\Gamma)$ and $p \in H^1(\Gamma) \cap L_0^2(\Gamma)$ and there exists a constant $C_2 > 0$ such that,
     \begin{equation}
     \label{eq8}
         \| \boldsymbol{u} \|_{\boldsymbol{H}^2_{\tan}(\Gamma)} + \| p \|_{H^1(\Gamma)} \leq C_2 \| \boldsymbol{f} \|_{\boldsymbol{L}^2(\Gamma)}.
     \end{equation}
\end{lemma}

\section{Non-conforming finite element approximation} \label{sec4}
\subsection{Surface approximation and associated results}\label{sec4.1}
Let $\Gamma_h$ denote a polyhedral approximation of the surface $\Gamma$, obtained through a triangulation $\mathcal{T}_h$ consisting of planar triangular elements $K$. The triangulation is performed such that the vertices lie on $\Gamma$ and $\Gamma_h$, where  $\Gamma_h$=$\bigcup\limits_{K \in \mathcal{T}_h} K$. The associated mesh size is defined by $h \!:=\underset{K \in \mathcal{T}_h}{\text{max}} \hspace{1mm} \text{diam}(K).$  $\mathcal{T}_h$ is assumed to be shape-regular and quasi-uniform. Let $\boldsymbol{n}_h^{K}$ denote the outward normal vector to $K $. Therefore, the projection operator onto the tangent space of the discrete surface $\Gamma_h$ is given by $$\textbf{P}_h= \textbf{I}-\boldsymbol{n}_h^{K} (\boldsymbol{n}_h^{K})^T.$$ 
For the sake of simplicity, we drop the superscript $K$ from now on.
\begin{figure}[htbp]
\centering
\begin{subfigure}{.45\textwidth}
  \centering
  \includegraphics[width=.8\linewidth]{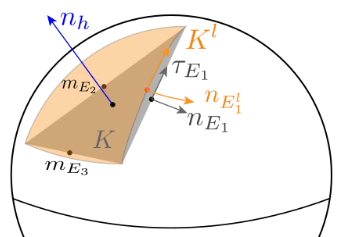}
\end{subfigure}%
\begin{subfigure}{.45\textwidth}
  \centering
  \includegraphics[width=.6\linewidth]{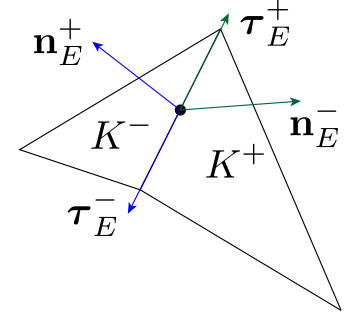}
\end{subfigure}
\caption{Left: Tangential plane coinciding with the sphere at the vertices of the triangle; Right: The conormal, $\boldsymbol{n}_E^\pm$ and tangential, $\boldsymbol{\tau}_E^{\pm}$ vectors sharing an edge $E$. }
\label{fig1}
\end{figure}
Let $\mathcal{E}_h$ denote the collection of all the interior edges and $E$ be an interior edge shared by two neighbouring triangles $K^+$ and $K^-$. Associated with each element $K^{\pm}$, we define the unit tangential vector $\boldsymbol{\tau}_E^{\pm}$ along the edge $E$, and the corresponding unit conormal vector $\boldsymbol{n}_E^{\pm}$, which is outward pointing, orthogonal to $E$, and lies in the same plane as $\boldsymbol{\tau}_{E}^{\pm}$. \textcolor{black}{Figure} \ref{fig1} illustrates the orientation of the vectors. Further details regarding their construction and properties can be found in \cite{mehlmann2023}.\vspace{1mm}\\ 
Using the closest-point projection function, each planar triangle $K \in \mathcal{T}_h$ can be mapped onto the exact surface $\Gamma$. We denote the resulting curved element by $K^l \!:$=$p(K)$ such that $\Gamma$=$\bigcup\limits_{K \in \mathcal{T}_h} p(K)$=$\bigcup\limits_{K^l \in \mathcal{T}_h^l} K^l$, where $\mathcal{T}_h^l= \{K^l \!: \! K \in \mathcal{T}_h  \} $ denotes the lifted triangulation of $\Gamma$. The tangential and conormal vectors on the discrete surface, $\boldsymbol{\tau}_E$ and $\boldsymbol{n}_E$ can be lifted analogously onto the exact surface, denoting them by $(\boldsymbol{\tau}_{E})^l$ and $(\boldsymbol{n}_{E})^l$, respectively. Moreover, $\boldsymbol{\tau}_{E^l}$ and $\boldsymbol{n}_{E^l}$ denote the tangential and conormal vectors on the exact curved  surface $\Gamma$. They have the following properties:
$$\boldsymbol{n}_{E^l}^{+}=-\boldsymbol{n}_{E^l}^{-} \hspace{2mm} \text{and} \hspace{2mm} \boldsymbol{\tau}_{E^l}^{+}=-\boldsymbol{\tau}_{E^l}^{-}.$$
In contrast, on the discrete surface $\Gamma_h$,
$$\boldsymbol{n}_{E}^{+} \neq - \boldsymbol{n}_{E}^{-} \hspace{2mm} \text{but} \hspace{2mm} \boldsymbol{\tau}_{E}^{+} = -\boldsymbol{\tau}_{E}^{-}.$$
These tangential and conormal vectors on $\Gamma_h$ play a fundamental role in the construction of the \textcolor{black}{ finite element space in Section} \ref{ss421}.
\begin{remark}
    Clearly, $\boldsymbol{\tau}_E$ and $\boldsymbol{n}_E$ are orthogonal, and so are $\boldsymbol{\tau}_{E^l}$ and $\boldsymbol{n}_{E^l}$, but $\boldsymbol{\tau}_E$ and $\boldsymbol{\tau}_{E^l}$, as well as $\boldsymbol{n}_E$ and $\boldsymbol{n}_{E^l}$ do not point in the same direction. Therefore, geometric errors $(\boldsymbol{\tau}_{E^l}- \boldsymbol{\tau}_{E})$ and $(\boldsymbol{n}_{E^l}- \boldsymbol{n}_{E})$ arise.
\end{remark}
In the following, we mention several geometric approximation results that will be used throughout the analysis. In this paper, we assume that $\Gamma_h \subset U_{\delta}$ and $ds$ and $ds_h$ denote the surface measures of $\Gamma$ and $\Gamma_h$, respectively. Following Larson et. al. in \cite{larsson2017} (see eq. (3.16)), for any $\boldsymbol{x} \in \Gamma_h$, there exists $\mu_h$  such that $ds \circ p= \mu_h ds_h$ holds and we obtain
\begin{equation}
    \label{m1}
     \|1-\mu_h\|_{L^{\infty}(K)}  \leq Ch^2. 
\end{equation}
Similarly, Burman et. al. showed in \cite{burman2019} (see eq. (3.9)) that for the edge measures $d \sigma$ and $d \sigma_h$, there exists $\mu_h^E$ such that $d\sigma \circ p= \mu_h^E d\sigma_h$ holds and similarly
\begin{equation}
     \label{m2}
     \|1-\mu_h^E\|_{L^{\infty}(E)}  \leq Ch^2,
\end{equation}
where $C$ is a generic constant independent of the mesh size $h$.
\begin{lemma}[Geometric approximation results]
\label{lem4} 
    Let $\Gamma_h \subset U_{\delta}$ be the polyhedral approximation of $\Gamma$ with the properties mentioned above and $\mathbf{P}^e$ and $\boldsymbol{n}_{E^l}^{e}$ denote the extension of $\mathbf{P}$ and $\boldsymbol{n}_{E^l}$, respectively, on $\Gamma_h$. For a sufficiently small mesh size $h$, the following results hold:
    \begin{align}
    \label{9a}
    \| \boldsymbol{n}^e- \boldsymbol{n}_h \|_{L^{\infty}(K)} & \leq C h, \\
        \label{eq9}
        \| \mathbf{P}^e- \mathbf{P}_h \|_{L^{\infty}(K)} & \leq C h, \\
        \label{eq10}
        \| \boldsymbol{n}_{E^l}^{e\pm} - \boldsymbol{n}_{E}^{\pm} \|_{L^{\infty}(E)} & \leq C h, \\
        \label{eq11}
        \| \boldsymbol{n}_{E^l}^{e\pm} - \mathbf{P}^e \boldsymbol{n}_{E}^{\pm} \|_{L^{\infty}(E)} & \leq C h^2, \\
         \label{eq12}
        \| \mathbf{P}_h \boldsymbol{n}_{E^l}^{e\pm} -  \boldsymbol{n}_{E}^{\pm} \|_{L^{\infty}(E)} & \leq C h^2.
    \end{align}
\end{lemma}
\begin{proof}
    The proof of (\ref{9a})-(\ref{eq11}) can be found in Appendix A of \cite{larsson2017} and the derivation of (\ref{eq12}) is given in Lemma 3.1 of \cite{mehlmann2023}.
\end{proof}

\begin{remark}[Geometric approximation]
   As noted in \cite{mehlmann2023}, the estimates (\ref{eq10})-(\ref{eq12}) also hold for the tangential vectors $\boldsymbol{\tau}_E$ and $\boldsymbol{\tau}_{E^l}$.
\end{remark}
Before proceeding further, here we introduce the broken norms defined on the spaces $\boldsymbol{L}^2(\Omega_h)$, $\boldsymbol{H}^m(\Omega_h)$ and $\boldsymbol{H}_{\text{tan}}^m (\Omega_h)$ for $\Omega_h \subset \Gamma_h$. For an element $K$ $\in \Omega_h$, we define
\begin{align*}
 \| \boldsymbol{w} \|_{\boldsymbol{L}^2(\Omega_h)}^2 & : = \sum_{K \in \Omega_h} \| \boldsymbol{w} \|_{\boldsymbol{L}^2(K)}, \\
    \| \boldsymbol{w} \|_{\boldsymbol{H}^m(\Omega_h)}^2 & : =  \sum_{i=0}^m \| (\nabla \mathbf{P}_h)^i \boldsymbol{w} \|^2_{\boldsymbol{L}^2(\Omega_h)},\\  \hspace{1mm} \text{and} \hspace{1mm} \| \boldsymbol{w} \|_{\boldsymbol{H}^m_{\text{tan}}(\Omega_h)}^2 & : = \sum_{i=0}^m \| (\nabla_{\Gamma_h})^i \boldsymbol{w} \|^2_{\boldsymbol{L}^2(\Omega_h)}.
\end{align*}
The broken semi-norm $| \cdot|_{\boldsymbol{H}^m(\Omega_h)}$ can be defined analogously. The norms and semi-norms defined on $\Gamma$ and $\Gamma_h$ are related through the following equivalence estimates:
\begin{lemma}[Norm equivalence results]
\label{lem5}
     Let $K \in \mathcal{T}_h$ be an element of $\Gamma_h$, and let $K^l \subset \Gamma$ denote its lift onto the exact surface. For any $\boldsymbol{w} \in \boldsymbol{H}^2(K^l)$ the following estimate holds:
    \begin{align}
    \label{eq22}
        \| \boldsymbol{w} \|_{\boldsymbol{L}^2(K^l)} & \leq C \| \boldsymbol{w}^e \|_{\boldsymbol{L}^2(K)} \leq C  \| \boldsymbol{w} \|_{\boldsymbol{L}^2(K^l)}, \\
        \label{eq23}
        |\boldsymbol{w} |_{\boldsymbol{H}^1(K^l)} & \leq C  |\boldsymbol{w}^e |_{\boldsymbol{H}^1(K)} \leq C  | \boldsymbol{w} |_{\boldsymbol{H}^1(K^l)}, \\
        \label{eq23a}
        | \boldsymbol{w} |_{\boldsymbol{H}^2(K^l)} & \leq C | \boldsymbol{w}^e |_{\boldsymbol{H}^2(K)},\\
        \label{eq24a}
         | \boldsymbol{w}^e |_{\boldsymbol{H}^2(K)} & \leq C  | \boldsymbol{w} |_{\boldsymbol{H}^2(K^l)}.
    \end{align}
\end{lemma}
\begin{proof}
 Hansbo et. al. established (\ref{eq22})-(\ref{eq23}) for vector-valued functions in \cite{HansboLarson2020} and the inequalities (\ref{eq23a})-(\ref{eq24a}) are proven in  Appendix B of \cite{larsson2017} for scalar-valued functions, which can be extended for vector-valued set up component-wise.
\end{proof}

Furthermore, we present the operator difference estimates and geometric estimates, which play a key role in the subsequent error analysis.
\begin{lemma}[Operator differences]\label{lemma:OpDiff}
For $ \boldsymbol{v} \in \boldsymbol{H}^1_{\tn}(\Gamma_h)$, we have 
\begin{equation}
\label{I8}
\| (\nabla_\Gamma \boldsymbol{v}^l)^e-\nabla_{\Gamma_h}\boldsymbol{v} \|_{\boldsymbol{L}^2(\Gamma_h)}\leq C h \|\boldsymbol{v} \|_{\boldsymbol{H}^1(\Gamma_h)}.
\end{equation}
For $\boldsymbol{w} \in \boldsymbol{H}^1_{\tn} (\Gamma)$, it holds that
\begin{equation}
    \label{8a}
    \| (\nabla_\Gamma \boldsymbol{w})^e-\nabla_{\Gamma_h}\boldsymbol{w}^e \|_{\boldsymbol{L}^2(\Gamma_h)}\leq C h \|\boldsymbol{w}^e \|_{\boldsymbol{H}^1(\Gamma_h)}.
\end{equation}
\end{lemma}
\begin{proof}
The proof of the result (\ref{I8}) follows from Lemma 5.4 in \cite{HansboLarson2020} with $k_g=1$ using $\boldsymbol{v}_{n_h}=0$ for $\boldsymbol{v} \in \boldsymbol{H}^1_{\tn}(\Gamma_h)$. \vspace{2mm}\\ 
Derivation of (\ref{8a}) proceeds as follows: for $\boldsymbol{w} \in \boldsymbol{H}^1_{\tn} (\Gamma)$, we have $\boldsymbol{w}^e \in \boldsymbol{H}^1 (\Gamma_h)$ and $\boldsymbol{w}^e= \boldsymbol{w}^e_{t_h} +\boldsymbol{w}^e_{n_h},$ where the tangential and normal parts are, respectively, $\boldsymbol{w}^e_{t_h} = \textbf{P}_h \boldsymbol{w}^e$ and $\boldsymbol{w}^e_{n_h} = (\textbf{I}-\textbf{P}_h) \boldsymbol{w}^e = (\textbf{P}+ \boldsymbol{n} \boldsymbol{n}^T-\textbf{P}_h) \boldsymbol{w}^e= (\textbf{P}-\textbf{P}_h) \boldsymbol{w}^e$. \vspace{1mm}\\
By Lemma 5.4 of \cite{HansboLarson2020} and later using the geometric approximation property (\ref{eq9}), we obtain
\begin{align*}
     \| (\nabla_\Gamma \boldsymbol{w})^e-\nabla_{\Gamma_h}\boldsymbol{w}^e \|_{\boldsymbol{L}^2(\Gamma_h)} & \leq C h \Big(  \|\boldsymbol{w}^e \|_{\boldsymbol{H}^1(\Gamma_h)} + h^{-1} \| \boldsymbol{w}^e_{n_h} \|_{\boldsymbol{L}^2(\Gamma_h)} \Big) \\
     & \leq C h \Big(  \|\boldsymbol{w}^e \|_{\boldsymbol{H}^1(\Gamma_h)} + h^{-1} \| \mathbf{P}^e- \mathbf{P}_h \|_{L^{\infty}(\Gamma_h)} \| \boldsymbol{w}^e \|_{\boldsymbol{L}^2(\Gamma_h)} \Big) \\
     & \leq Ch \|\boldsymbol{w}^e \|_{\boldsymbol{H}^1(\Gamma_h)}.
\end{align*}

\end{proof}

\begin{lemma}[Geometric $H^1$-estimate] \label{G1}
For $\boldsymbol{u}, \boldsymbol{v} \in \boldsymbol{H}^2_{\tn}(\Gamma)$, the following result holds:
\begin{equation}
    \label{I8a}
    | (\nabla_{\Gamma_h} \boldsymbol{u}^e, \nabla_{\Gamma_h} \boldsymbol{v}^e)_{\Gamma_h}- (\nabla_\Gamma \boldsymbol{u},\nabla_\Gamma \boldsymbol{v})_\Gamma| \leq C h^2  \| \boldsymbol{u} \|_{\boldsymbol{H}^2_{\tn}(\Gamma)} \| \boldsymbol{v} \|_{\boldsymbol{H}^2_{\tn}(\Gamma)}.
\end{equation}
\end{lemma}

\begin{proof}
 \textcolor{black}{   The estimate (\ref{I8a}) can be proven by applying the estimate (5.33) of Lemma 5.5 in \cite{HansboLarson2020} with $k_g=1$.}
\end{proof}

\begin{lemma}[Geometric error-estimate for the divergence operator]\label{lemdiv} For $(\boldsymbol{v},q) \in \boldsymbol{H}^2_{\tn}(\Gamma) \times H^1(\Gamma)$, we have
\begin{equation}
    \label{I9}
   \left| (q, \textnormal{div}_{\Gamma} \boldsymbol{v}- (\textnormal{div}_{\Gamma_h} \boldsymbol{v}^e)^l)_{\Gamma} \right| \leq C h^2 \| \boldsymbol{v} \|_{\boldsymbol{H}^2_{\tn}(\Gamma)} \| q \|_{H^1(\Gamma)}.
\end{equation}
    
\end{lemma}
\begin{proof}
    The proof can be found in the proof of Lemma 4.8 in \cite{hardering2025}.
\end{proof}

\begin{lemma}[Geometric $L^2$-estimate]\label{I9}
    For $\boldsymbol{f}, \boldsymbol{v} \in \boldsymbol{L}^2(\Gamma)$, the following estimate holds:
    \begin{equation}
        |(\boldsymbol{f}, \boldsymbol{v})_{\Gamma}- (\boldsymbol{f}^e, \boldsymbol{v}^e)_{\Gamma_h}| \leq C h^2 \| \boldsymbol{f} \|_{\boldsymbol{L}^2(\Gamma)} \| \boldsymbol{v} \|_{\boldsymbol{L}^2(\Gamma)}.
    \end{equation}
\end{lemma}
\begin{proof}
    We first apply a change of variable and then the Cauchy-Schwarz inequality. Later, employing the result (\ref{m1}) and the norm equivalence (\ref{eq22}) subsequently, we obtain the desired estimate.
    \begin{align*}
        |(\boldsymbol{f}, \boldsymbol{v})_{\Gamma}- (\boldsymbol{f}^e, \boldsymbol{v}^e)_{\Gamma_h}| & \leq  \|1-\mu_h\|_{L^{\infty}(\Gamma_h)} \| \boldsymbol{f}^e \|_{\boldsymbol{L}^2(\Gamma_h)} \| \boldsymbol{v}^e \|_{\boldsymbol{L}^2(\Gamma_h)} \\
        & \leq C h^2  \| \boldsymbol{f} \|_{\boldsymbol{L}^2(\Gamma)} \| \boldsymbol{v} \|_{\boldsymbol{L}^2(\Gamma)}. 
    \end{align*}
\end{proof}

\subsection{The Crouzeix-Raviart finite element formulation}
\subsubsection{Tangential vector-valued Crouzeix-Raviart space}
\label{ss421}
A detailed discussion on the construction of the Crouzeix–Raviart finite element space for vector-valued functions can be found in \cite{mehlmann2023}. Here, we briefly summarize the main aspects of the formulation relevant to the present work. \vspace{1mm}\\
Let $\hat{K}$ denote the reference triangle with unit edge lengths. For each element $K \in \mathcal{T}_h$, we define an affine mapping $F_K: \hat{K} \rightarrow K$, which maps the reference element onto the physical element $K$. The Crouzeix–Raviart basis functions  $\{ \hat{\phi}_i \}_{i=1}^3$ on $\hat{K}$ are defined using the midpoints $m_{\hat{E}_i}$ of the edges $\hat{E}_i$ as follows:
$$\hat{\phi}_i \in P^1(\hat{K}), \hspace{2mm} \hat{\phi}_i(m_{\hat{E}_j})= \delta_{ij} \hspace{2mm} \forall i,j=1,2,3.$$ Using the edge tangential and conormal vectors, we introduce the finite-dimensional space
$$V_K:= \text{span} \{\boldsymbol{n}_{E_i} \phi_i^K, \boldsymbol{\tau}_{E_j} \phi_j^K, i,j=1,2,3 \},$$
where $\phi_i^K= \hat{\phi}_i \circ F_K^{-1}$. \vspace{1mm}\\
By construction, all elements of $V_K$ are tangential to $\Gamma_h$.
Using this construction, we define in the following the space of tangential Crouzeix–Raviart functions on $\Gamma_h$ whose tangential and conormal components are continuous across edge midpoints:
\begin{equation}
    \begin{split} \label{J0}
        \textbf{V}_h= \{\boldsymbol{v}_h: \Gamma_h \rightarrow \mathbb{R}^3| & \,\boldsymbol{v}_h |_K \in V_K, \hspace{1mm} \forall E = \partial K^{+} \cap \partial K^{-}, \\
& \quad \boldsymbol{v}_h |_{K^+} (m_E) \cdot \boldsymbol{n}_E^+ = - \boldsymbol{v}_h |_{K^-} (m_E) \cdot \boldsymbol{n}_E^- , \\
& \quad \quad \boldsymbol{v}_h |_{K^+} (m_E) \cdot \boldsymbol{\tau}_E^+ = - \boldsymbol{v}_h |_{K^-} (m_E) \cdot \boldsymbol{\tau}_E^- \}.
    \end{split}
\end{equation}
Now using the mid-point rule, we obtain 
\begin{equation}
\label{J1}
    \int_E \lJump \boldsymbol{v}_h \cdot \boldsymbol{n}_E \rJump \hspace{1mm} d \sigma_h = 0 \hspace{1mm} \text{and} \hspace{1mm} \int_E \lJump \boldsymbol{v}_h \cdot \boldsymbol{\tau}_E \rJump \, d \sigma_h = 0,
\end{equation}
where the jump is defined as 
\begin{equation}
    \label{J2}
    \lJump \boldsymbol{v} \rJump = \underset{s \rightarrow 0^+}{\lim} (\boldsymbol{v}(\boldsymbol{x}-s \boldsymbol{n}_{E^+})-\boldsymbol{v}(\boldsymbol{x}-s \boldsymbol{n}_{E^-})).
\end{equation}
The discrete vector-valued Crouzeix-Raviart space, $\textbf{V}_h$, is combined with the space of the piecewise constant functions, $Q_h=\{ q_h \in L_0^2(\Gamma_h): q_h |_K \textcolor{black}{\in P^0(K)} \}$ to discretize the Stokes problem.

\subsubsection{Discrete formulation}
For each element $K$, the discrete derivative and the strain rate tensor of a vector-valued function $\boldsymbol{v}_h \in \textbf{V}_h$ are, respectively, defined by 
 $$\nabla_{\Gamma_h} \boldsymbol{v}_h = \textbf{P}_h \nabla \boldsymbol{v}_h \textbf{P}_h \hspace{2mm} \text{and} \hspace{2mm} E_{\Gamma_h}(\boldsymbol{v}_h) = \frac{1}{2} \left( \nabla_{\Gamma_h} \boldsymbol{v}_h + (\nabla_{\Gamma_h} \boldsymbol{v}_h)^T \right).$$ 
The finite element formulation consists of finding $ (\boldsymbol{u}_h, p_h) \in (\textbf{V}_h, Q_h)$ such that 
\begin{equation}
\begin{split}
\label{eq13}
       a_h(\boldsymbol{u}_h, \boldsymbol{v}_h) + b_h(\boldsymbol{v}_h,p_h)  & = l(\boldsymbol{v}_h) \hspace{2mm} \forall  \boldsymbol{v}_h \in \textbf{V}_h,  \\
        b_h(\boldsymbol{u}_h,q_h)  & = 0 \hspace{8mm} \forall q_h \in Q_h,
\end{split}
\end{equation}
where
\begin{align}
    a_h(\boldsymbol{u}_h, \boldsymbol{v}_h) & = \sum_{K \in \mathcal{T}_h} \int_K \left[ E_{\Gamma_h}(\boldsymbol{u}_h) :  E_{\Gamma_h}(\boldsymbol{v}_h)  +  \boldsymbol{u}_h \cdot \boldsymbol{v}_h \right] \, ds_h + \sum_{E \in \mathcal{E}_h} \int_E \frac{\alpha}{h} \lJump \boldsymbol{u}_h \rJump \lJump \boldsymbol{v}_h \rJump \, d\sigma_h, \\
    b_h(\boldsymbol{v}_h,p_h) & = \sum_{K \in \mathcal{T}_h} \int_K p_h (\text{div}_{\Gamma_h} \boldsymbol{v}_h) \, ds_h,\\
    l(\boldsymbol{v}_h) & = \int_{\Gamma_h} \boldsymbol{f}^e \cdot \boldsymbol{v}_h \, ds_h,
\end{align}
where $\boldsymbol{f}^e$ is the extension of $\boldsymbol{f}$ to $\Gamma_h$.
To study the well-posedness of the discrete formulation, we introduce the discrete energy norm $\| \cdot \|_h$ on $\textbf{V}_h$ defined by
\begin{equation}
\label{eq18}
    \|\boldsymbol{v}_h \|_h^2 = \sum_{K \in \mathcal{T}_h} \| \nabla_{\Gamma_h} \boldsymbol{v}_h \|_{\boldsymbol{L}^2(K)}^2 + \| \boldsymbol{v}_h\|_{\boldsymbol{L}^2(\Gamma_h)}^2+ \sum_{E \in \mathcal{E}_h} \frac{\alpha}{h} \|\lJump \boldsymbol{v}_h \rJump \|_{\boldsymbol{L}^2(E)}^2, \hspace{2mm} \forall \hspace{1mm} \boldsymbol{v}_h \in \textbf{V}_h. 
\end{equation}
\textcolor{black}{The norm $\| \cdot \|_h$ defines a complete norm on $\textbf{V}_h$. The continuity of the bilinear form $a_h(\cdot, \cdot)$ can be established applying the triangle inequality and the Cauchy-Schwarz inequality. The bilinear term $b_h(\cdot, \cdot)$ is shown to be continuous using the Cauchy-Schwarz inequality and the following result
$$ \| \text{div}_{\Gamma_h} \boldsymbol{u}_h \|_{L^2(\Gamma_h)}= \| \text{tr}(\nabla_{\Gamma_h}\boldsymbol{u}_h) \|_{L^2(\Gamma_h)} \leq C \|\nabla_{\Gamma_h}\boldsymbol{u}_h \|_{\boldsymbol{L}^2(\Gamma_h)}. $$
The coercive property of $a_h(\cdot, \cdot)$ is implied from the discrete Korn's inequality, which can be established following the ideas of Brenner \cite{brenner04} and Bonito et. al. \cite{bonito2020} and is the subject of an upcoming work. Therefore, to establish the existence and uniqueness of the discrete solution, it remains to verify that the pair of finite element spaces $(\textbf{V}_h, Q_h)$ satisfies the discrete inf-sup stability condition. This property is established in the following section.}

\subsubsection{Discrete inf-sup stability result}
Prior to establishing the discrete inf-sup stability condition, we introduce an interpolation operator $\tilde{\Pi}_h^{\text{tan}}\!:\!\boldsymbol{H}^1_{\text{tan}} (\Gamma) \rightarrow \textbf{V}_h$ that will be instrumental in the subsequent analysis. It is a weighted variant of the interpolation operator introduced by Mehlmann in \cite{mehlmann2023}. The construction of $\tilde{\Pi}_h^{\tn}$ is summarized in the following. \vspace{1mm}\\
Let $\{ \phi \}_{i=1}^3$ denote the local Crouzeix–Raviart basis functions associated with an element $K \in \Gamma_h$. For $i=1,...,3$, $\boldsymbol{\tau}_{E_i}$ and $\boldsymbol{n}_{E_i}$ denote the tangential and conormal vectors to the edges $E_i$ of $K$. $\boldsymbol{\tau}_{E_i^l}$ and $\boldsymbol{n}_{E_i^l}$ are the same quantities defined on the lifted counterpart $E_i^l$ on $\Gamma$. Furthermore, for a tangential vector field $\boldsymbol{v} \in \boldsymbol{H}^1_{\text{tan}} (\Gamma)$, the quantities $$(\boldsymbol{v} \cdot \boldsymbol{n}_{E_i^l})^e \hspace{2mm} \text{and} \hspace{2mm} (\boldsymbol{v} \cdot \boldsymbol{\tau}_{E^l_i})^e$$ represent the extensions from $K^l$ to $K$ of the conormal and tangential components of $\boldsymbol{v}$, respectively. Now the operator is defined elementwise as
\begin{equation}
\label{eq32}
\tilde{\Pi}_h^{\text{tan}} \boldsymbol{v} \mid_K = \sum_{i=1}^3 \tilde{\alpha}_i^{\text{tan}}(\boldsymbol{v}) \phi_i,    
\end{equation}
\text{where} the coefficient $$\tilde{\alpha}_i^{\text{tan}}(\boldsymbol{v}) = \frac{1}{| E_i |} \int_{E_i} \left[ (\boldsymbol{v} \cdot \boldsymbol{n}_{E^l_i})^e \boldsymbol{n}_{E_i} + (\boldsymbol{v} \cdot \boldsymbol{\tau}_{E^l_i})^e \boldsymbol{\tau}_{E_i} \right] \mu_h^{E_i} \hspace{1mm} d \sigma_h.$$
The definition yields the following properties on an edge $E$:
\begin{equation}
\label{eq33}
    \frac{1}{| E |} \int_{E} \tilde{\Pi}_h^{\text{tan}} \boldsymbol{v} \cdot \boldsymbol{n}_{E} \hspace{1mm}  d\sigma_h  = \frac{1}{| E |} \int_{E} (\boldsymbol{v} \cdot \boldsymbol{n}_{E^l})^e \hspace{1mm} \mu_h^{E} \hspace{1mm} d\sigma_h,
\end{equation}

\begin{equation}
\label{eq34}
    \frac{1}{| E |} \int_{E} \tilde{\Pi}_h^{\text{tan}} \boldsymbol{v} \cdot \boldsymbol{\tau}_{E} \hspace{1mm} d\sigma_h  = \frac{1}{| E |} \int_{E} (\boldsymbol{v} \cdot \boldsymbol{\tau}_{E^l})^e \hspace{1mm} \mu_h^{E} \hspace{1mm} d\sigma_h.
\end{equation}
The property (\ref{eq33}) is essential for establishing the following intermediate results (Lemma \ref{lem6} and Lemma \ref{lem5a}), which play a central role in the proof of the inf-sup stability estimate in Theorem \ref{thm1}.

\begin{lemma}[Conormal jump preservation property]
\label{lem6}
    For any $\boldsymbol{v} \in \boldsymbol{H}^1_{\tn}(\Gamma)$, the following holds
    
    \begin{equation}
    \label{eq29}
        \int_E \lJump (\tilde{\Pi}_h^{\tn} \boldsymbol{v} \cdot \boldsymbol{n}_E) q_h \rJump \hspace{1mm} d \sigma_h= \int_E \lJump (\boldsymbol{v}^e \cdot \boldsymbol{n}_{E^l}^e) q_h \rJump \hspace{1mm} \mu_h^E \, d \sigma_h,
    \end{equation}
    for any $q_h \in Q_h$.
\end{lemma}
\begin{proof}
    Let, $\boldsymbol{v}_h=\tilde{\Pi}_h^{\tn} \boldsymbol{v} \in \textbf{V}_h$. \vspace{1mm}\\
First, we expand the jump term and since $q_h$ is piecewise constant on $\Gamma_h$, the traces $q_h^+$ and $q_h^-$ appear. Then, applying property (\ref{eq33}), we obtain the expression in the third line, which leads to the final result.
    \begin{align*}
        \int_E \lJump (\boldsymbol{v}_h \cdot \boldsymbol{n}_E ) q_h \rJump \hspace{1mm} d \sigma_h & = \int_E \left( \boldsymbol{v}_h^+ \cdot \boldsymbol{n}_E^+ q_h^+ + \boldsymbol{v}_h^- \cdot \boldsymbol{n}_E^- q_h^-\right) \hspace{1mm} d \sigma_h \\
        & = q_h^+ \mid_E  \int_E \boldsymbol{v}_h^+ \cdot \boldsymbol{n}_E^+ \hspace{1mm} d \sigma_h + q_h^- \mid_E  \int_E \boldsymbol{v}_h^- \cdot \boldsymbol{n}_E^- \hspace{1mm} d \sigma_h \\
       & = q_h^+ \mid_E  \int_E \boldsymbol{v}^{e+} \cdot \boldsymbol{n}_{E^l}^{e+} \hspace{1mm} \mu_h^E \, d \sigma_h +  q_h^- \mid_E  \int_E \boldsymbol{v}^{e-} \cdot \boldsymbol{n}_{E^l}^{e-}\hspace{1mm} \mu_h^E \, d \sigma_h \\
        & = \int_E \lJump \boldsymbol{v}^e \cdot \boldsymbol{n}_{E^l}^{e} q_h \rJump \hspace{1mm} \mu_h^E \, d \sigma_h.
    \end{align*}
\end{proof}

\begin{lemma}[Divergence preservation under lifting] \label{lem5a}
     For any $\boldsymbol{v} \in \boldsymbol{H}^1_{\tn}(\Gamma)$ and any $q_h \in Q_h$, the following \textcolor{black}{equality} holds:
     \begin{equation}
         \label{eq29aa}
         b_h(\tilde{\Pi}_h^{\tn} \boldsymbol{v}, q_h)= b(\boldsymbol{v}, q_h^l).
     \end{equation}
\end{lemma}

\begin{proof}
We begin the derivation with employing integration by parts elementwise. Since $q_h \in Q_h$, we have $\nabla_{\Gamma_h} q_h |_K =0$, and hence the elementwise integration gives rise to jump contributions across the edges. We obtain the expression in the third step by employing the conormal jump preservation property (\ref{eq29}) of the interpolation operator. Next, a change of variables is used to lift the integral onto the exact surface. Using again the property $\nabla_{\Gamma_h} q_h^l |_K =0$, we obtain the desired expression. 
\allowdisplaybreaks
    \begin{align*}
        b_h(\tilde{\Pi}_h^{\tn} \boldsymbol{v},q_h) 
            & = \sum_{K \in \mathcal{T}_h} \int_K \left(\text{div}_{\Gamma_h} \tilde{\Pi}_h^{\tn} \boldsymbol{v} \right) q_h  \hspace{1mm} ds_h  \\
            & = \sum_{E \in \mathcal{E}_h} \int_E \lJump ( \tilde{\Pi}_h^{\tn} \boldsymbol{v} \cdot \boldsymbol{n}_E ) q_h \rJump \hspace{1mm}  d \sigma_h   \\
            & = \sum_{E \in \mathcal{E}_h} \int_E \lJump (\boldsymbol{v}^e \cdot \boldsymbol{n}_{E^l}^e )q_h \rJump \hspace{1mm} \mu_h^E d \sigma_h \numberthis \label{eq36} \\
            & = \sum_{E^l \in \mathcal{E}_h^l} \int_{E^l} \lJump (\boldsymbol{v} \cdot \boldsymbol{n}_{E^l} ) q_h^l \rJump \hspace{1mm} d \sigma \\
            & =  \int_{\Gamma} (\text{div}_{\Gamma} \boldsymbol{v}) (q_h)^l \hspace{1mm} ds.
    \end{align*}
\end{proof}


\begin{theorem}[Discrete inf-sup condition]
\label{thm1}
    There exists a constant $C>0$, independent of $h$, such that for all $q_h \in Q_h$,
    \begin{equation} \label{1s}
        \underset{\boldsymbol{v}_h \in \textbf{V}_h}{\mathrm{sup}} \frac{b_h(\boldsymbol{v}_h, q_h)}{ \|\boldsymbol{v}_{h}\|_{h}} \geq C \| q_h\|_{L^2(\Gamma_h)}.
    \end{equation}
\end{theorem}
\begin{proof} 
For a given discrete pressure, we first construct a suitable discrete velocity test function that satisfies the stability and compatibility properties required for the inf-sup analysis. \vspace{1mm}\\
\textcolor{black}{For a fixed $q_h \in Q_h$,} there exists $\boldsymbol{w} \in \boldsymbol{H}^1_{\tn}(\Gamma)$ (Pg. 258 in \cite{demlow2024}) such that 
\begin{equation}
\begin{split}
\label{eq30}
    \text{div}_{\Gamma} \boldsymbol{w} & =  q_h^l \in L^2(\Gamma), \hspace{1mm} \text{and} \\
    \|\boldsymbol{w} \|_{\boldsymbol{H}^1(\Gamma)} & \leq C \|  q_h^l \|_{L^2(\Gamma)} \leq C \| q_h \|_{L^2(\Gamma_h)}.
\end{split}
\end{equation}
 We choose the discrete test function $\boldsymbol{v}_h$ as $$\boldsymbol{v}_h = \tilde{\Pi}_h^{\tn} \boldsymbol{w}.$$ 
We now derive a lower bound for $b_h(\boldsymbol{v}_h,q_h)$ in terms of $\| q_h \|_{L^2(\Gamma_h)}^2$. This estimate provides the numerator, required in the inf-sup condition. \vspace{1mm} \\
The divergence preservation property (\ref{eq29aa}) of the interpolation operator $\tilde{\Pi}_h^{\tn}$ gives the first step in the following derivation. Later, the estimate (\ref{eq30}) is invoked in the last line and we achieve the appropriate estimate by applying the norm equivalence result (\ref{eq22}).
\allowdisplaybreaks
    \begin{equation}
    \begin{split} \label{30a}
        b_h(\boldsymbol{v}_h,q_h) 
            & =  \int_{\Gamma} (\text{div}_{\Gamma} \boldsymbol{w}) (q_h)^l  \hspace{1mm} ds \\
            & =  \|q_h^l \|_{L^2(\Gamma)}^2 \geq C \| q_h \|_{L^2(\Gamma_h)}^2.
            \end{split}
    \end{equation}
  Afterwards, we establish an upper bound for the denominator term $\| \boldsymbol{v}_h \|_h$ with respect to $\| q_h\|_{L^2(\Gamma_h)}$ to obtain the compatibility result (\ref{1s}). \vspace{1mm}\\
 To estimate the broken $H^1$-norm part of the energy norm $\| \cdot \|_h$, we apply the $H^1$-stability of the interpolation operator $\tilde{\Pi}_h^{tan}$ \textcolor{black}{in the second inequality and the result (\ref{eq30}) in the last bound}
    \begin{equation}
    \begin{split}
    \label{36a}
        \sum_{K \in \mathcal{T}_h} \| \nabla_{\Gamma_h} \boldsymbol{v}_h \|_{\boldsymbol{L}^2(K)}^2 + \| \boldsymbol{v}_h\|_{\boldsymbol{L}^2(\Gamma_h)}^2 & \leq C \| \boldsymbol{w} \|_{\boldsymbol{H}^1(\Gamma)}^2 \leq C \| q_h \|_{L^2(\Gamma_h)}.
    \end{split}
    \end{equation}
   Next, we derive a suitable estimate for the jump term, which appears in $\| \cdot \|_h$. To this end, we introduce the edge average $$\bar{\boldsymbol{v}}_h= \frac{1}{| E |} \int_E \boldsymbol{v}_h \hspace{1mm} d \sigma_h.$$
By definition of the vector-valued Crouzeix-Raviart element, $\lJump \bar{\boldsymbol{v}}_h\rJump=0$, see (\ref{J0}). We apply the trace inequality and the discrete Poincar$\acute{e}$ inequality 
successively in the \textcolor{black}{third and fourth steps}, respectively. \textcolor{black}{Besides, $\nabla_{\Gamma_h} \bar{\boldsymbol{v}}_h=0 $ is used to obtain the bound in fourth step.} The last two \textcolor{black}{bounds} are obtained employing the $H^1$-stability and the result (\ref{eq30}), respectively.
\allowdisplaybreaks
    \begin{align*}
         \frac{1}{h} \|\lJump \boldsymbol{v}_h \rJump \|_{\boldsymbol{L}^2(E)}^2 & = \frac{1}{h} \|\lJump \boldsymbol{v}_h - \bar{\boldsymbol{v}}_h\rJump \|_{\boldsymbol{L}^2(E)}^2 \\
   & \leq  \frac{1}{h} \left( \| \boldsymbol{v}_h - \bar{\boldsymbol{v}}_h  \|_{\boldsymbol{L}^2(\partial K^+)}^2 + \| \boldsymbol{v}_h - \bar{\boldsymbol{v}}_h  \|_{\boldsymbol{L}^2(\partial K^-)}^2\right) \\
    & \leq \frac{C}{h} \left[ h^{-1} \| \boldsymbol{v}_h - \bar{\boldsymbol{v}}_h \|_{\boldsymbol{L}^2(K^+)}^2 + h  \| \nabla_{\Gamma_h} (\boldsymbol{v}_h - \bar{\boldsymbol{v}}_h) \|_{\boldsymbol{L}^2(K^+)}^2 \right] + \\
    & \quad \frac{C}{h} \left[ h^{-1} \| \boldsymbol{v}_h - \bar{\boldsymbol{v}}_h \|_{\boldsymbol{L}^2(K^-)}^2 + h  \| \nabla_{\Gamma_h} (\boldsymbol{v}_h - \bar{\boldsymbol{v}}_h) \|_{\boldsymbol{L}^2(K^-)}^2 \right] \numberthis \label{36b}\\
    & \leq \sum_{K \in \mathcal{T}_h} \frac{C }{h} \left[ h^{-1} h^2 \| \nabla_{\Gamma_h} \boldsymbol{v}_h \|^2_{\boldsymbol{L}^2(K)} + h  \| \nabla_{\Gamma_h} \boldsymbol{v}_h \|_{\boldsymbol{L}^2(K)}^2\right] \hspace{1mm} \\
    & \leq C \sum_{K \in \mathcal{T}_h} \| \nabla_{\Gamma_h} \boldsymbol{v}_h \|^2_{\boldsymbol{L}^2(K)}  \leq C \| \boldsymbol{w} \|_{\boldsymbol{H}^1(\Gamma)}^2 \leq  C \| q_h\|_{L^2(\Gamma_h)}^2.
\end{align*}
 Hence, \textcolor{black}{combining (\ref{36b})} with (\ref{36a}), we obtain
\begin{equation}
\label{eq24}
    \| \boldsymbol{v}_h \|_h \leq C \| q_h \|_{L^2(\Gamma_h)}.
\end{equation}
\textcolor{black}{Using the lower bound for $b(\boldsymbol{v}_h, q_h)$ obtained in (\ref{30a}) together with the estimate} (\ref{eq24}) for $\| \boldsymbol{v}_h \|_{h}$, we obtain the desired discrete inf-sup condition with a constant independent of the mesh size $h$ as follows:
\[ \frac{b_h(\boldsymbol{v}_h, q_h)}{\| \boldsymbol{v}_h \|_{h}} \geq C  \frac{ \| q_h \|_{L^2(\Gamma_h)}^2}{ \| q_h \|_{L^2(\Gamma_h)}}  = C \| q_h \|_{L^2(\Gamma_h)}.  \]
\end{proof}

\section{Error estimation} \label{sec5}
\subsection{Interpolation estimates}
The following estimates play a key role in the subsequent error analysis.
\begin{lemma}[Interpolation estimates]
\label{lem7}
For $\boldsymbol{v} \in \boldsymbol{H}^1_{\tn} (\Gamma_h) \cap \boldsymbol{H}^2(\Gamma_h)$ and for the interpolation operator $\tilde{\Pi}_h^{\tn}$ introduced in (\ref{eq32}), the following interpolation error estimates hold:
    \begin{align}
    \label{I1}
        \| \nabla_{\Gamma_h} (\boldsymbol{v}^e- \tilde{\Pi}_h^{\tn} \boldsymbol{v}) \|_{\boldsymbol{L}^2(\Gamma_h)} & \leq C h \| \boldsymbol{v}^e \|_{\boldsymbol{H}^2(\Gamma_h)}, \\
        \label{I2}
         \| \mathbf{P}^e (\boldsymbol{v}^e- \tilde{\Pi}_h^{\tn} \boldsymbol{v}) \|_{\boldsymbol{L}^2(\Gamma_h)} & \leq C h^2 \| \boldsymbol{v}^e \|_{\boldsymbol{H}^2(\Gamma_h)}, \\
         \label{I3}
         \| \mathbf{P}^e (\boldsymbol{v}^e- \tilde{\Pi}_h^{\tn} \boldsymbol{v}) \|_{\boldsymbol{L}^2(E)} & \leq C h^{\frac{3}{2}} \| \boldsymbol{v}^e \|_{\boldsymbol{H}^2(\Gamma_h)}, \\
         \label{I4}
         \| \mathbf{P}_h (\boldsymbol{v}^e- \tilde{\Pi}_h^{\tn} \boldsymbol{v}) \|_{\boldsymbol{L}^2(\Gamma_h)} & \leq C h^2 \| \boldsymbol{v}^e \|_{\boldsymbol{H}^2(\Gamma_h)}, \\
         \label{I5}
         \| \mathbf{P}_h (\boldsymbol{v}^e- \tilde{\Pi}_h^{\tn} \boldsymbol{v}) \|_{\boldsymbol{L}^2(E)} & \leq C h^{\frac{3}{2}} \| \boldsymbol{v}^e \|_{\boldsymbol{H}^2(\Gamma_h)}, 
    \end{align}
    and in the energy norm, we have
    \begin{equation}
    \label{I6}
         \| \boldsymbol{v}^e- \tilde{\Pi}_h^{\tn} \boldsymbol{v} \|_{h}  \leq C h \| \boldsymbol{v}^e \|_{\boldsymbol{H}^2(\Gamma_h)}.
    \end{equation}
\end{lemma}
\begin{proof}
We observe that the interpolation operator $\tilde{\Pi}_h^{\tn}$ is a weighted variant of the interpolation estimate $\Pi_h^{tan}$, introduced in \cite{mehlmann2023},
\begin{align*}
   \Pi_h^{tan}\boldsymbol{v}|_K & = \sum^3_{i=1}  \phi_i\alpha_i^{tan}(\boldsymbol{v}^e),\\ 
   & \alpha_i^{tan}(\boldsymbol{v}^e)=\frac{1}{|E_i|}\int_{E_i} ( \boldsymbol{v} \cdot \boldsymbol{n}_{E_i^l})^e \boldsymbol{n}_{E_i} \, d\sigma_h + \frac{1}{|E_i|}\int_{E_i} ( \boldsymbol{v} \cdot \boldsymbol{\tau}_{E_i^l})^e  \boldsymbol{\tau}_{E_i} \, d\sigma_h. 
\end{align*}
Consequently, the derivations of the results (\ref{I1})-(\ref{I5}) follow the same lines as those in Lemma 5.1 of \cite{mehlmann2023}, where the weight is bounded by applying the result (\ref{m2}).\vspace{2mm}\\
\textbf{Derivation of (\ref{I6}):} \textcolor{black}{ Let $\boldsymbol{w}= (\boldsymbol{v}^e- \tilde{\Pi}_h^{\tn} \boldsymbol{v})$. Applying (\ref{I1}) on $\| \nabla_{\Gamma_h} \boldsymbol{w} \|_{\boldsymbol{L}^2(K)}$ and (\ref{I4}) on $\| \boldsymbol{w} \|_{\boldsymbol{L}^2(\Gamma_h)}$, we obtain
{\allowdisplaybreaks
\begin{align}
\label{31a}
   \|\boldsymbol{w} \|_h^2  
    & \leq C h^2 \| \boldsymbol{v}^e \|_{\boldsymbol{H}^2(\Gamma_h)}^2 + \sum_{E \in \mathcal{E}_h} \frac{\alpha}{h} \|\lJump \boldsymbol{w} \rJump \|_{\boldsymbol{L}^2(E)}^2.
\end{align}
}
To estimate the jump term, let $\bar{\boldsymbol{w}}= \frac{1}{\mid E \mid} \int_E \boldsymbol{w} \hspace{1mm} d \sigma_h$ denote the edge average of $\boldsymbol{w}$. By (\ref{J0}), we have $\lJump \bar{\boldsymbol{w}}\rJump=0$. Applying steps 2-5 used in derivation (\ref{36b}) yields the first bound. Further application of (\ref{I1}) leads to the final estimate.
\begin{equation}
\begin{split}
\label{31b}
    \frac{1}{h} \|\lJump \boldsymbol{w} \rJump \|_{\boldsymbol{L}^2(E)}^2 & = \frac{1}{h} \|\lJump \boldsymbol{w} - \bar{\boldsymbol{w}}\rJump \|_{\boldsymbol{L}^2(E)}^2 \\
    & \leq C \| \nabla_{\Gamma_h} (\boldsymbol{v}^e- \tilde{\Pi}_h^{\tn} \boldsymbol{v}) \|_{\boldsymbol{L}^2(\Gamma_h)}^2 \\
    & \leq C h^2 \| \boldsymbol{v}^e \|_{\boldsymbol{H}^1(\Gamma_h)}^2.
    \end{split}
\end{equation}
Hence, combining (\ref{31a}) and (\ref{31b}), we obtain the desired result. }
\end{proof}
For the pressure variable, we introduce the following interpolation operator. Let $I_h: L^2(\Gamma) \rightarrow Q_h$ denote the $L^2$-projection operator, defined by $$(I_h p)\mid_K = \frac{1}{| K |} \int_K p^e \mu_h \hspace{1mm} ds_h. $$
Moreover, we recall the corresponding interpolation estimate from \cite{li2014} (Theorem 3.6),
\begin{lemma}[Interpolation estimate]
   For any $p \in H^1(\Gamma)$, the following interpolation estimate holds 
   \begin{equation}
\label{I7}
    \| p^e - I_h p \|_{L^2(\Gamma_h)} \leq C h \| p^e \|_{H^1(\Gamma_h)}.
\end{equation}
\end{lemma}

\subsection{Energy-error estimation}
The coercivity of the bilinear form $a_h(\boldsymbol{u}_h, \boldsymbol{v}_h)$ together with the inf-sup compatibility condition for the discrete pair, implies that the discrete formulation (\ref{eq13}) admits a unique solution $\boldsymbol{u}_h \in \textbf{V}_h$ and $p_h \in Q_h$. These coercive and inf-sup stability properties of the bilinear terms further ensure the validity of the \textcolor{black}{Second Strang Lemma} [\cite{li2014}, eq. (2.11)], which is subsequently used to derive the \textcolor{black}{a-priori} error estimate in the energy norm. 
\begin{theorem}[Strang's Second Lemma]
\label{thm2}
 Let $(\boldsymbol{u},p) \in \boldsymbol{H}^2_{\tn}(\Gamma) \times H^1(\Gamma) \cap L_0^2(\Gamma)$ be the solution of (2) and $(\boldsymbol{u}^e,p^e)$ be their extensions. For the discrete solution $(\boldsymbol{u}_h,p_h) \in \mathbf{V}_h \times Q_h$ satisfying (\ref{eq13}), the following estimate holds:
   \begin{equation}
        \begin{split}
        \label{eq45}
            \|\boldsymbol{u}^e- \boldsymbol{u}_h \|_h + \|p^e-p_h \|_{L^2(\Gamma_h)} & \leq C \Bigg( \underbrace{  \underset{\boldsymbol{v}_h \in \mathbf{V}_h}{\mathrm{inf}}  \|\boldsymbol{u}^e- \boldsymbol{v}_h \|_h +  \underset{q_h \in Q_h}{\mathrm{inf}} \| p^e-q_h \|_{L^2(\Gamma_h)} }_{\text{approximation error}}\\
    & \quad + \underbrace{\underset{\boldsymbol{v}_h \in \mathbf{V}_h}{\mathrm{sup}} \frac{\mid a_h(\boldsymbol{u}^e, \boldsymbol{v}_h) + b_h(\boldsymbol{v}_h, p^e)- (\boldsymbol{f}^e, \boldsymbol{v}_h)  \mid}{\| \boldsymbol{v}_h \|_h}}_{\text{nonconformity error}}  \Bigg).
        \end{split}
    \end{equation}
\end{theorem}
The approximation error is estimated directly using the interpolation estimates (\ref{I6}) and (\ref{I7}), as will be shown in Theorem \ref{thm3}. The primary difficulty lies in the estimation of the nonconformity error, which requires the control of several geometric error terms as well as jump contributions. In the following lemma, we present the estimation for the nonconformity error, whose proof relies on the auxiliary results established subsequently in this section.
\begin{lemma}[Nonconformity Error Estimate] \label{lem12}
     Let $(\boldsymbol{u},p) \in \boldsymbol{H}^2_{\tn}(\Gamma) \times H^1(\Gamma) \cap L_0^2(\Gamma)$ be the solution of (2) and let $(\boldsymbol{u}^e,p^e)$ be their extensions. For $\boldsymbol{f} \in \boldsymbol{L}^2_{\text{tan}}(\Gamma)$ and any $\boldsymbol{v}_h \in \mathbf{V}_h$, the following error estimate holds:
    \begin{equation}
    \label{eq53}
       | a_h(\boldsymbol{u}^e, \boldsymbol{v}_h) + b_h(\boldsymbol{v}_h, p^e)- (\boldsymbol{f}^e, \boldsymbol{v}_h) | \leq C h  \| \boldsymbol{f} \|_{\boldsymbol{L}^2(\Gamma)} \| \boldsymbol{v}_h \|_h,
    \end{equation}
    where $\boldsymbol{f}^e$ is the extension of $\boldsymbol{f}$ to $\Gamma_h$.
\end{lemma}

\begin{proof}
    Upon adding and subtracting $(\boldsymbol{f}, \boldsymbol{v}_h^l)$ in the second line, replacing $\boldsymbol{f}$ by $( -  \textbf{P} \text{div}_{\Gamma} (E_{\Gamma} (\boldsymbol{u})) + \boldsymbol{u} - \nabla_{\Gamma} p)$ in the third line, and applying the Green's formula to the third expression in the third line, we obtain
   \begin{align*}
        & a_h(\boldsymbol{u}^e, \boldsymbol{v}_h) + b_h(\boldsymbol{v}_h, p^e)- (\boldsymbol{f}^e, \boldsymbol{v}_h) \\
    & = a_h(\boldsymbol{u}^e, \boldsymbol{v}_h) + b_h(\boldsymbol{v}_h, p^e)-(\boldsymbol{f}, \boldsymbol{v}_h^l)+ \underbrace{(\boldsymbol{f}, \boldsymbol{v}_h^l)- (\boldsymbol{f}^e, \boldsymbol{v}_h)}_{\text{T}_1} \\
    & = a_h(\boldsymbol{u}^e, \boldsymbol{v}_h) + b_h(\boldsymbol{v}_h, p^e) - ( -  \textbf{P} \text{div}_{\Gamma} (E_{\Gamma} (\boldsymbol{u})) + \boldsymbol{u} - \nabla_{\Gamma} p, \boldsymbol{v}_h^l) + \text{T}_1 \\
    & = \text{T}_1 + \underbrace{a_h(\boldsymbol{u}^e, \boldsymbol{v}_h)  - a(\boldsymbol{u}, \boldsymbol{v}_h^l)}_{\text{T}_2} + \underbrace{b_h(\boldsymbol{v}_h, p^e)- b(\boldsymbol{v}_h^l, p)}_{\text{T}_3} +  \\
    & \quad  \underbrace{ \sum_{E^l} \int_{E^l} \left( E_{\Gamma}(\boldsymbol{u}) \boldsymbol{n}_{E^l}^+ \cdot \boldsymbol{v}_h^{l+}  + E_{\Gamma}(\boldsymbol{u}) \boldsymbol{n}_{E^l}^- \cdot \boldsymbol{v}_h^{l-}  \right) \hspace{1mm} d \sigma }_{\text{T}_4} + \numberthis \label{eq54} \\
    & \quad  \underbrace{ \sum_{E^l} \int_{E^l} p \left(  \boldsymbol{n}_{E^l}^+ \cdot \boldsymbol{v}_h^{l+}  + \boldsymbol{n}_{E^l}^- \cdot \boldsymbol{v}_h^{l-}  \right) \hspace{1mm} d \sigma }_{\text{T}_5}  = \sum_{i=1}^5 \text{T}_i.
   \end{align*}
\textcolor{black}{The estimates for the geometric error terms T$_i$, for $i$=1,2,3, are provided in Lemma \ref{I9}, equation (\ref{eq51}) of Lemma \ref{lem10} and equation (\ref{eq52}) of Lemma \ref{lem11}, respectively.}
Employing $\boldsymbol{n}_{E^l}^+ =- \boldsymbol{n}_{E^l}^- $, we observe that the terms $\text{T}_4$ and $\text{T}_5$ contain jump terms as follows:
\begin{align*}
    \text{T}_4 &= \sum_{E^l \in \mathcal{E}_h^l}\int_{E^l} E_{\Gamma}(\boldsymbol{u}) \boldsymbol{n}_{E^l}^+ \cdot \lJump \boldsymbol{v}_h^l \rJump \hspace{1mm} d \sigma \\
     & = \frac{1}{2} \sum_{E^l \in \mathcal{E}_h^l} \int_{E^l} (\nabla_{\Gamma} \boldsymbol{u}) \boldsymbol{n}_{E^l}^+ \cdot \lJump \boldsymbol{v}_h^l \rJump \hspace{1mm} d \sigma + \frac{1}{2} \sum_{E^l \in \mathcal{E}_h^l}\int_{E^l} (\nabla_{\Gamma} \boldsymbol{u})^T \boldsymbol{n}_{E^l}^+ \cdot \lJump \boldsymbol{v}_h^l \rJump \hspace{1mm} d \sigma,\\
    \text{T}_5 & = \sum_{E^l \in \mathcal{E}_h^l} \int_{E^l} p \boldsymbol{n}_{E^l}^+ \cdot \lJump \boldsymbol{v}_h^l \rJump \hspace{1mm} d \sigma.
\end{align*}
The presence of $(\nabla_{\Gamma} \boldsymbol{u})^T$ does not affect the derivation in a substantial way, since $\|\nabla_{\Gamma} \boldsymbol{u}\|_{\boldsymbol{L}^2(\Gamma)} = \| (\nabla_{\Gamma} \boldsymbol{u})^T \|_{\boldsymbol{L}^2(\Gamma)}$. Hence, it suffices to estimate one of the representative terms in $\text{T}_4$. Lemma \ref{lem9} provides the estimates of the jump terms in T$_4$ and T$_5$. 
\textcolor{black}{Combining the estimates for T$_1$ to T$_5$ in (\ref{eq54}) yields (\ref{eq53}).}
\end{proof}
\begin{lemma}[Geometric error of the momentum equation] \label{lem10} 
    Let $\boldsymbol{u} \in \boldsymbol{H}^2_{\tn}(\Gamma)$ be the velocity solution of (2) and $\boldsymbol{u}^e$ be its extension. Then, for any $\boldsymbol{v}_h \in \mathbf{V}_h$, the following estimation holds:
    \begin{equation}
    \label{eq51}
       | a_h(\boldsymbol{u}^e, \boldsymbol{v}_h) - a(\boldsymbol{u}, \boldsymbol{v}_h^l) | \leq C h \|\boldsymbol{f}\|_{\boldsymbol{L}^2(\Gamma)} \| \boldsymbol{v}_h \|_h.
    \end{equation}
    Furthermore, for $\boldsymbol{w} \in \boldsymbol{H}^2_{\tn}(\Gamma)$, with $\boldsymbol{w}^e$ denoting its extension, the following estimate holds:
    \begin{equation}
    \label{eq51a}
       | a_h(\boldsymbol{u}^e, \boldsymbol{w}^e) - a(\boldsymbol{u}, \boldsymbol{w}) | \leq C h^2 \|\boldsymbol{u}\|_{\boldsymbol{H}^2_{\tn}(\Gamma)} \| \boldsymbol{w} \|_{\boldsymbol{H}^2_{\tn}(\Gamma)}.
    \end{equation}
\end{lemma}
\begin{proof}
We begin the derivation with expanding the bilinear terms. By continuity, $\lJump \boldsymbol{u}^e \rJump = 0$, and hence the jump term in the second line vanishes. In the third line, we first apply change of variables to the terms defined on the exact surface $\Gamma$, then subsequently introducing the terms $\int_K (E_{\Gamma} (\boldsymbol{u}))^e : E_{\Gamma_h} (\boldsymbol{v}_h) \hspace{1mm} ds_h$ and $\int_K (E_{\Gamma} (\boldsymbol{u}))^e : (E_{\Gamma} (\boldsymbol{v}_h^l))^e \hspace{1mm} ds_h$ and further rearranging the resulting expressions, we obtain
\begin{equation}
\begin{split}
\label{50a}
     & a_h(\boldsymbol{u}^e, \boldsymbol{v}_h) - a(\boldsymbol{u}, \boldsymbol{v}_h^l) \\
    & = \sum_{K \in \mathcal{T}_h} \int_K E_{\Gamma_h} (\boldsymbol{u}^e): E_{\Gamma_h} (\boldsymbol{v}_h) \hspace{1mm} ds_h+ \int_{\Gamma_h} \boldsymbol{u}^e \cdot \boldsymbol{v}_h \hspace{1mm} ds_h- \sum_{K^l \in \mathcal{T}_h^l} \int_{K^l} E_{\Gamma} (\boldsymbol{u}): E_{\Gamma} (\boldsymbol{v}_h^l)\hspace{1mm} ds   \\
    & \quad -\int_{\Gamma} \boldsymbol{u} \cdot \boldsymbol{v}_h^l\hspace{1mm} ds + \underbrace{\sum_{E \in \mathcal{E}_h} \int_E \frac{\alpha}{h_E} \lJump \boldsymbol{u}^e \rJump \lJump \boldsymbol{v}_h \rJump \hspace{1mm} d \sigma_h}_{=0} \\
    & =   \underbrace{ \sum_{K \in \mathcal{T}_h} \int_K \left[ E_{\Gamma_h} (\boldsymbol{u}^e)- (E_{\Gamma} (\boldsymbol{u}))^e \right]: E_{\Gamma_h} (\boldsymbol{v}_h) \hspace{1mm} ds_h}_{\text{T}_6} + \\
    & \qquad  \underbrace{ \sum_{K \in \mathcal{T}_h} \int_K (E_{\Gamma} (\boldsymbol{u}))^e : \left[ E_{\Gamma_h} (\boldsymbol{v}_h)- (E_{\Gamma} (\boldsymbol{v}_h^l))^e \right] \hspace{1mm} ds_h}_{\text{T}_7} + \\
    & \qquad   \underbrace{\sum_{K \in \mathcal{T}_h} \int_K (E_{\Gamma} (\boldsymbol{u}))^e: (E_{\Gamma} (\boldsymbol{v}_h^l))^e (1-\mu_h) \hspace{1mm} ds_h}_{\text{T}_8} + \underbrace{ \int_{\Gamma_h} (1-\mu_h) \boldsymbol{u}^e \cdot \boldsymbol{v}_h \hspace{1mm} ds_h. }_{\text{T}_9}
    \end{split}
    \end{equation}
The estimations of the terms $\text{T}_6$ and $\text{T}_7$ are analogous. Hence, it suffices to estimate T$_6$ below. Using the definition of the strain rate tensors, we obtain the second line using the commutative property of the transpose and extension operators.
Subsequently applying the Cauchy-Schwarz inequality, the operator difference result (\ref{8a}), the norm equivalence results (\ref{eq22})-(\ref{eq23}), and the regularity estimate (\ref{eq8}), we obtain
\begin{align*}
    \text{T}_6 & = \frac{1}{2} \sum_{K \in \mathcal{T}_h} \int_K \left[ \{ \nabla_{\Gamma_h} \boldsymbol{u}^e- (\nabla_{\Gamma} \boldsymbol{u})^e \}  + \{ (\nabla_{\Gamma_h} \boldsymbol{u}^e)^T- ((\nabla_{\Gamma} \boldsymbol{u})^T)^e \} \right] : E_{\Gamma_h} (\boldsymbol{v}_h) \hspace{1mm} ds_h \\
    & = \frac{1}{2} \sum_{K \in \mathcal{T}_h} \int_K \left[ \{ \nabla_{\Gamma_h} \boldsymbol{u}^e- (\nabla_{\Gamma} \boldsymbol{u})^e \} + \{ \nabla_{\Gamma_h} \boldsymbol{u}^e- (\nabla_{\Gamma} \boldsymbol{u})^e\}^T\right] : E_{\Gamma_h} (\boldsymbol{v}_h) \hspace{1mm} ds_h  \\
    & \leq C \| \nabla_{\Gamma_h} \boldsymbol{u}^e- (\nabla_{\Gamma} \boldsymbol{u})^e \|_{\boldsymbol{L}^2(\Gamma_h)} \| \boldsymbol{v}_h \|_{\boldsymbol{H}^1(\Gamma_h)} \numberthis \label{eq54c} \\
    & \leq C h \| \boldsymbol{u} \|_{\boldsymbol{H}^1(\Gamma)} \| \boldsymbol{v}_h \|_h\\ 
    & \leq C h \| \boldsymbol{f} \|_{\boldsymbol{L}^2(\Gamma)} \| \boldsymbol{v}_h \|_{h}.  
\end{align*}
Proceeding in the same manner, we obtain an analogous estimate for $\text{T}_7$. The remaining terms, $\text{T}_8$ and $\text{T}_9 $, are then estimated by successively applying the Cauchy-Schwarz inequality, the bound (\ref{m1}) for $(1-\mu_h)$, the norm equivalence stated in Lemma \ref{lem5}, and the regularity estimate (\ref{eq8}).
\begin{align*}
    \text{T}_8 & \leq \| 1-\mu_h \|_{L^{\infty}(\Gamma_h)} \| \boldsymbol{u} \|_{\boldsymbol{H}^1(\Gamma)} \| \boldsymbol{v}_h \|_{\boldsymbol{H}^1(\Gamma_h)} \leq C h^2 \| \boldsymbol{f} \|_{\boldsymbol{L}^2(\Gamma)} \| \boldsymbol{v}_h \|_{h},\\
    \text{T}_9 & \leq  \| 1-\mu_h \|_{L^{\infty}(\Gamma_h)} \| \boldsymbol{u}^e \|_{\boldsymbol{L}^2(\Gamma_h)}  \| \boldsymbol{v}_h \|_{\boldsymbol{L}^2(\Gamma_h)} \leq C h^2 \| \boldsymbol{f} \|_{\boldsymbol{L}^2(\Gamma)} \| \boldsymbol{v}_h \|_{h}.
\end{align*}
Combining the estimates for T$_6$-T$_9$ in (\ref{50a}), we have (\ref{eq51}). \vspace{2mm}\\
\textcolor{black}{The estimate in (\ref{eq51a}) is obtained by combining the arguments used in the proof of (\ref{eq54c}) with the geometric estimate given in Lemma \ref{G1}.}
\end{proof}

\begin{lemma}[Geometric error of mass conservation]\label{lem11}
    Let $p \in H^1(\Gamma) \cap L_0^2(\Gamma)$ be the solution of (2), $p^e$ be its extension and $\boldsymbol{v}_h^l$ denote the lift of $\boldsymbol{v}_h$. Then, the following estimate associated with the bilinear term holds for any $\boldsymbol{v}_h \in \mathbf{V}_h$:
    \begin{equation}
    \label{eq52}
       | b_h(\boldsymbol{v}_h, p^e)- b(\boldsymbol{v}_h^l, p) | \leq C h \|\boldsymbol{f}\|_{\boldsymbol{L}^2(\Gamma)} \| \boldsymbol{v}_h \|_h.
    \end{equation}
     Furthermore, for $\boldsymbol{u} \in \boldsymbol{H}^2_{\tn}(\Gamma)$ and $q \in H^1(\Gamma) \cap L_0^2(\Gamma) $, the following estimate holds: 
    \begin{equation}
    \label{eq65}
       |b_h(\boldsymbol{u}^e, q^e)- b(\boldsymbol{u}, q)| \leq C h^2 \|\boldsymbol{u}\|_{\boldsymbol{H}_{\tn}^2(\Gamma)} \| q \|_{H^1(\Gamma)},
    \end{equation}
    where $\boldsymbol{u}^e$ and $q^e$ are their extensions, respectively.
\end{lemma}
\begin{proof} Upon adding and subtracting $\int_{\Gamma_h} (\text{div}_{\Gamma} \boldsymbol{v}_h^l)^e p^e\hspace{1mm} ds_h$, we obtain
 \begin{equation}
     \begin{split}
         \label{eq52d}
          &  b_h(\boldsymbol{v}_h, p^e)- b(\boldsymbol{v}_h^l, p) \\
    & = \underbrace{\int_{\Gamma_h} \left( \text{div}_{\Gamma_h} \boldsymbol{v}_h - (\text{div}_{\Gamma} \boldsymbol{v}_h^l)^e \right) p^e \hspace{1mm} ds_h }_{\text{T}_{10}}+ \underbrace{\left[ \int_{\Gamma_h} (\text{div}_{\Gamma} \boldsymbol{v}_h^l)^e p^e \hspace{1mm} ds_h-  \int_{\Gamma} (\text{div}_{\Gamma} \boldsymbol{v}_h^l) p \hspace{1mm} ds \right].}_{\text{T}_{11}} 
     \end{split}
 \end{equation}
As the trace operator and the extension operator commute, we obtain the first step in the following estimation of $\text{T}_{10}$ by applying the Cauchy-Schwarz inequality. Later, employing the operator difference result (\ref{I8}) and the norm equivalence (\ref{eq22}), we obtain
\begin{equation}
    \begin{split}
    \label{eq52a}
       \text{T}_{10} & \leq \| \text{tr} (\nabla_{\Gamma_h} \boldsymbol{v}_h) - \text{tr}((\nabla_{\Gamma} \boldsymbol{v}_h^l)^e) \|_{L^2(\Gamma_h)} \| p^e \|_{L^2(\Gamma_h)} \\
    & \leq C \| \nabla_{\Gamma_h} \boldsymbol{v}_h - (\nabla_{\Gamma} \boldsymbol{v}_h^l)^e \|_{\boldsymbol{L}^2(\Gamma_h)} \| p^e \|_{L^2(\Gamma_h)} \\
    & \leq C h \| \boldsymbol{v}_h \|_{h} \| p \|_{L^2(\Gamma)}.  
    \end{split}
\end{equation}
To estimate $\text{T}_{11}$, we apply a change of variable on the term defined on the exact surface $\Gamma$. We then successively employ the estimate (\ref{m1}), the Cauchy-Schwarz inequality, and the operator difference result (\ref{I8}) after introducing $\nabla_{\Gamma_h} \boldsymbol{v}_h$. Finally, applying a further change of variable to the pressure term yields the following estimate:
\begin{equation}
    \begin{split}
    \label{eq52c}
   \text{T}_{11}
   & \leq \| 1-\mu_h \|_{L^{\infty}(\Gamma_h)} \int_{\Gamma_h} (\text{div}_{\Gamma} \boldsymbol{v}_h^l)^e p^e \hspace{1mm} ds_h \\
   & \leq C h^2 \| (\nabla_{\Gamma} \boldsymbol{v}_h^l)^e \|_{\boldsymbol{L}^2(\Gamma_h)} \| p^e \|_{L^2(\Gamma_h)} \\
   & \leq C h^2 \Big[ \| (\nabla_{\Gamma} \boldsymbol{v}_h^l)^e -\nabla_{\Gamma_h} \boldsymbol{v}_h \|_{\boldsymbol{L}^2(\Gamma_h)} + \| \nabla_{\Gamma_h} \boldsymbol{v}_h \|_{\boldsymbol{L}^2(\Gamma_h)} \Big] \| p^e \|_{L^2(\Gamma_h)} \\
   & \leq C h^2 \| \boldsymbol{v}_h \|_{h} \| p \|_{L^2(\Gamma)}. 
\end{split}
\end{equation}
Combining the above estimates (\ref{eq52a})-(\ref{eq52c}) in (\ref{eq52d}) and invoking the regularity result (\ref{eq8}), we obtain the desired result (\ref{eq52}). \vspace{2mm}\\
\textcolor{black}{The bound (\ref{eq65}) is obtained by combining the steps of the proof of (\ref{eq52a}) with the geometric bound provided in Lemma \ref{lemdiv}.}
\end{proof}

\begin{lemma}[Jump terms]
\label{lem9}
    Let $(\boldsymbol{u},p) \in \boldsymbol{H}^2_{\tn}(\Gamma) \times H^1(\Gamma) \cap L_0^2(\Gamma)$ be the solution of (2) and $(\boldsymbol{u}^e,p^e)$ be their extensions. Then the following error estimations hold:
    \begin{align}
    \label{eq48}
          \left| \sum_{E^l \in \mathcal{E}_h^l} \int_{E^l} (\nabla_{\Gamma} \boldsymbol{u}) \boldsymbol{n}_{E^l}^+ \cdot \lJump \boldsymbol{v}_h^l \rJump \right| & \leq C h \| \boldsymbol{v}_h \|_h \| \boldsymbol{f} \|_{\boldsymbol{L}^2(\Gamma)}, \\
          \label{eq49}
            \left| \sum_{E^l \in \mathcal{E}_h^l} \int_{E^l} p \boldsymbol{n}_{E^l}^+ \cdot \lJump \boldsymbol{v}_h^l \rJump \right| & \leq C h \| \boldsymbol{v}_h \|_h \| \boldsymbol{f} \|_{\boldsymbol{L}^2(\Gamma)}.
    \end{align}
\end{lemma}

\begin{proof}
The proof of (\ref{eq48}) can be found in Lemma 5.4 of \cite{mehlmann2023}. Proceeding analogously for the pressure variable, we obtain (\ref{eq49}).
\end{proof}

\begin{theorem}[Energy-error estimate] \label{thm3}
    Let $(\boldsymbol{u}^e,p^e) $ be the extension of the solution $(\boldsymbol{u},p) \in \boldsymbol{H}^2_{\tn}(\Gamma) \times H^1(\Gamma) \cap L_0^2(\Gamma)$ of (\ref{eq2}) and $(\boldsymbol{u}_h,p_h)$ the discrete solution of (\ref{eq13}). Then the following error estimate holds 
    \begin{equation}
    \label{eq55}
        \|\boldsymbol{u}^e- \boldsymbol{u}_h \|_h + \|p^e-p_h \|_{L^2(\Gamma_h)} \leq C h \| \boldsymbol{f} \|_{\boldsymbol{L}^2(\Gamma)}.
    \end{equation}
\end{theorem}
\begin{proof}
    The nonconformity estimate provided in Lemma \ref{lem12} yields
    \begin{equation}
    \label{eq56}
        \underset{\boldsymbol{v}_h \in \textbf{V}_h}{\text{sup}} \frac{| a_h(\boldsymbol{u}^e, \boldsymbol{v}_h) + b_h(\boldsymbol{v}_h, p^e)- (\boldsymbol{f}^e, \boldsymbol{v}_h)  |}{\| \boldsymbol{v}_h \|_h} \leq C h \| \boldsymbol{f} \|_{\boldsymbol{L}^2(\Gamma)}.
    \end{equation}
    We next address the approximation errors, for which we employ the estimates (\ref{I6}), (\ref{I7}) and the regularity result (\ref{eq8}). This gives
    \begin{align}
    \label{eq57}
        \underset{\boldsymbol{v}_h \in \textbf{V}_h}{\text{inf}}  \|\boldsymbol{u}^e- \boldsymbol{v}_h \|_h   \leq \|\boldsymbol{u}^e - \tilde{\Pi}_h^{tan} \boldsymbol{u} \|_h & \leq C h \| \boldsymbol{u}^e \|_{\boldsymbol{H}^2(\Gamma_h)} \leq C h \| \boldsymbol{f} \|_{\boldsymbol{L}^2(\Gamma)}, \\
        \label{eq58}
         \underset{q_h \in Q_h}{\text{inf}} \| p^e-q_h \|_{L^2(\Gamma_h)}  \leq \| p^e -I_h p \|_{L^2(\Gamma_h)} & \leq C h \| p^e \|_{H^1(\Gamma_h)} \leq C h \| \boldsymbol{f} \|_{\boldsymbol{L}^2(\Gamma)}.
    \end{align}
    The proof is concluded by combining the Strang's \textcolor{black}{Second} Lemma (\ref{eq45}) with the bounds (\ref{eq56})-(\ref{eq58}).
\end{proof}

\subsection{$L^2$-error estimate}
The $L^2$-error estimate is derived using the Aubin–Nitsche duality argument. The associated dual problem is: find $\boldsymbol{z} \in \boldsymbol{H}^1_{\tn}(\Gamma)$ and $w \in L^2_0(\Gamma)$ such that 
 \begin{equation}
        \begin{split}
        \label{eq59}
             a(\boldsymbol{z}, \boldsymbol{v}) + b(\boldsymbol{v},w) & = (\boldsymbol{g}, \boldsymbol{v})_{\Gamma} \hspace{4mm} \forall \boldsymbol{v} \in \boldsymbol{H}^1_{\tn}(\Gamma), \\
    b(\boldsymbol{z},q) & = 0 \hspace{13mm} \forall q \in L_0^2(\Gamma),
        \end{split}
    \end{equation}
holds,  where $\boldsymbol{g}= \textbf{P}_h^l(\boldsymbol{u}-\boldsymbol{u}_h^l)$. The following regularity estimate \textcolor{black}{\cite{olshanskii2018}} holds true:
\begin{equation}
\label{eq60}
         \| \boldsymbol{z} \|_{\boldsymbol{H}^2_{\tn}(\Gamma)} + \| w \|_{H^1(\Gamma)} \leq C \| \boldsymbol{g} \|_{\boldsymbol{L}^2(\Gamma)}.
     \end{equation}
The corresponding discrete dual problem to (\ref{eq59}) is to find $\boldsymbol{z}_h \in \textbf{V}_h$ and $w_h \in Q_h$ such that
\begin{equation}
\begin{split}
\label{eq61}
       a_h(\boldsymbol{z}_h, \boldsymbol{v}_h) + b_h(\boldsymbol{v}_h,w_h)  & = (\boldsymbol{g}^e, \boldsymbol{v}_h) \hspace{2mm} \forall \boldsymbol{v}_h \in \textbf{V}_h, \\
        b_h(\boldsymbol{z}_h,q_h)  & = 0 \hspace{13mm} \forall q_h \in Q_h.
\end{split}
\end{equation}
The following energy estimate for the dual solution follows from the energy error estimate provided in Theorem \ref{thm3}.
\begin{equation}
\label{eq62}
   \|\boldsymbol{z}^e- \boldsymbol{z}_h \|_h + \|w^e-w_h \|_{L^2(\Gamma_h)} \leq C h \| \boldsymbol{g} \|_{\boldsymbol{L}^2(\Gamma)}. 
\end{equation}
Based on Lemma \ref{lem12}, the following nonconformity error estimate holds:
\begin{equation}
\label{eq63}
       \left|a_h(\boldsymbol{z}^e, \boldsymbol{v}_h) + b_h(\boldsymbol{v}_h, w^e)- (\boldsymbol{g}^e, \boldsymbol{v}_h) \right| \leq C h \| \boldsymbol{v}_h \|_h \| \boldsymbol{g} \|_{\boldsymbol{L}^2(\Gamma)}.
    \end{equation}
The derivation of the $L^2$-error estimate \textcolor{black}{requires an approximation of} the nonconforming velocity and the geometric consistency errors arising from the surface approximation. Furthermore, the duality argument introduces interpolation and consistency errors associated with both the primal and dual problems, whose estimates must be combined to recover the optimal convergence order. To address these challenges, we make essential use of the pressure-coupling estimates, the dual interpolation estimate, and the primal and dual consistency error estimates established in the following \textcolor{black}{Lemma \ref{lem18} - Lemma \ref{lem17}, respectively.} 

\begin{lemma}[\textcolor{black}{Pressure-coupling estimates}] \label{lem18}
    Let $(\boldsymbol{u}^e,p^e)$ and $(\boldsymbol{z}^e,w^e)$ be the extensions of the solutions $(\boldsymbol{u},p) \in \boldsymbol{H}^2_{\tn}(\Gamma) \times H^1(\Gamma) \cap L_0^2(\Gamma)$ of (\ref{eq2}) and $(\boldsymbol{z},w) \in \boldsymbol{H}^2_{\tn}(\Gamma) \times H^1(\Gamma) \cap L_0^2(\Gamma)$ of the dual problem (\ref{eq59}), respectively. \textcolor{black}{For $(\boldsymbol{u}_h, p_h)$ the discrete solution of (\ref{eq13}) and $(\boldsymbol{z}_h, w_h)$ of (\ref{eq61}),} following estimates hold:
    \begin{align}
    \label{eq72}
         | b_h(\boldsymbol{z}^e, p_h-p^e) | & \leq C h^2 \| \boldsymbol{f} \|_{\boldsymbol{L}^2(\Gamma)} \| \boldsymbol{g} \|_{\boldsymbol{L}^2(\Gamma)}, \\
         \label{eq73}
         | b_h(\boldsymbol{u}^e, w_h-w^e) | & \leq C h^2 \| \boldsymbol{f} \|_{\boldsymbol{L}^2(\Gamma)} \| \boldsymbol{g} \|_{\boldsymbol{L}^2(\Gamma)}.
    \end{align}
\end{lemma}
\begin{proof}
    The two inequalities can be proved analogously; therefore, we only derive the second one. Introducing the term $\int_{\Gamma_h} (\text{div}_{\Gamma} \boldsymbol{u})^e (w_h-w^e ) \hspace{1mm} ds_h$ in the second step, \textcolor{black}{applying a change of variables, and adding and subtracting $\int_{\Gamma} (\text{div}_{\Gamma} \boldsymbol{u})^e (w_h-w^e ) \mu_h \hspace{1mm} ds_h$, we get}
    \begin{equation}
    \begin{split} \label{e70}
       &  b_h(\boldsymbol{u}^e, w_h-w^e)  \\
         & = \underbrace{\int_{\Gamma_h} \left[ \text{div}_{\Gamma_h} \boldsymbol{u}^e - (\text{div}_{\Gamma} \boldsymbol{u})^e \right] (w_h-w^e ) \hspace{1mm} ds_h }_{\text{I}_{1}} + \underbrace{\int_{\Gamma_h} (\text{div}_{\Gamma} \boldsymbol{u})^e (w_h-w^e ) (1-\mu_h) \hspace{1mm} ds_h }_{\text{I}_{2}} \\
         & \quad + \underbrace{\int_{\Gamma} (\text{div}_{\Gamma} \boldsymbol{u}) (w_h^l-w ) \hspace{1mm} ds.}_{\text{I}_{3}} 
         \end{split}
    \end{equation}
   \textcolor{black}{ To estimate $\text{I}_1$, we subsequently apply the Cauchy-Schwarz inequality and the following result in the first step.
    \begin{equation} 
    \label{74b}
       \| \text{div}_{\Gamma_h} (\boldsymbol{u}^e-\boldsymbol{u}_h) \|_{L^2(\Gamma_h)}= \| \text{tr}(\nabla_{\Gamma_h}(\boldsymbol{u}^e-\boldsymbol{u}_h)) \|_{L^2(\Gamma_h)} \leq C \|\nabla_{\Gamma_h}(\boldsymbol{u}^e-\boldsymbol{u}_h) \|_{\boldsymbol{L}^2(\Gamma_h)}. 
    \end{equation} 
    The second bound results by applying the operator difference (\ref{8a}) and the energy error estimate (\ref{eq62}). Employing the regularity estimate (\ref{eq8}) yields the final bound.}
    \begin{equation}
    \begin{split} \label{e71}
        \text{I}_{1} 
        & \leq C \| \nabla_{\Gamma_h} \boldsymbol{u}^e- ( \nabla_{\Gamma} \boldsymbol{u})^e \|_{\boldsymbol{L}^2(\Gamma_h)} \| w_h - w^e \|_{L^2(\Gamma_h)} \\
        & \leq C h^2 \| \boldsymbol{u}^e \|_{\boldsymbol{H}^2(\Gamma_h)} \| \boldsymbol{g} \|_{\boldsymbol{L}^2(\Gamma)} \\
        & \leq C h^2 \| \boldsymbol{f} \|_{\boldsymbol{L}^2(\Gamma)} \| \boldsymbol{g} \|_{\boldsymbol{L}^2(\Gamma)}.
     \end{split}
    \end{equation}
    We estimate $\text{I}_{2}$ by applying the Cauchy-Schwarz inequality, the result (\ref{m1}), the regularity estimate (\ref{eq8}), and the energy error estimate (\ref{eq62}) subsequently.
    \begin{equation}
    \begin{split} \label{e72}
        \text{I}_{2} & \leq \| 1-\mu_h \|_{L^{\infty}(\Gamma_h)} \| (\text{div}_{\Gamma} \boldsymbol{u})^e \|_{\boldsymbol{L}^2(\Gamma_h)} \| w_h - w^e \|_{L^2(\Gamma_h)} \\
        & \leq C h^3 \| \boldsymbol{f} \|_{\boldsymbol{L}^2(\Gamma)} \| \boldsymbol{g} \|_{\boldsymbol{L}^2(\Gamma)}.
     \end{split}
    \end{equation}
    To estimate the last term $\text{I}_{3}$, we introduce the mean value of $w_h^l$, $\bar{w}_h^l = \frac{1}{|\Gamma|} \int_{\Gamma} w_h^l \hspace{1mm} ds$ and obtain
    \begin{equation}
    \begin{split} \label{e73}
        \text{I}_{3} & = \int_{\Gamma} (\text{div}_{\Gamma} \boldsymbol{u}) (w_h^l-\bar{w}_h^l-w) \hspace{1mm} ds +  \int_{\Gamma} (\text{div}_{\Gamma} \boldsymbol{u}) \bar{w}_h^l \, ds 
     \end{split}
    \end{equation}
   \textcolor{black}{ As $(w_h^l-\bar{w}_h^l-w) \in L^2_0(\Gamma)$, the weak formulation (\ref{eq2}) implies $b(\boldsymbol{u}, w_h^l-\bar{w}_h^l-w)=0$. Therefore, the first term in $\text{I}_{3}$ vanishes. By applying the divergence theorem and the fact that $\lJump \boldsymbol{u} \cdot \boldsymbol{n}_{E^l} \rJump$=0, we get }
    \begin{equation}
    \begin{split} \label{e74}
        \text{I}_{3}   = \sum_{E^l \in \mathcal{E}_h^l} \int_{E^l} \bar{w}_h^l \lJump \boldsymbol{u} \cdot \boldsymbol{n}_{E^l} \rJump  =0.
     \end{split}
    \end{equation}
    Combining all the estimates (\ref{e71})-(\ref{e74}) in (\ref{e70}), we finally reach at (\ref{eq73}).
\end{proof}

\begin{lemma} [Dual interpolation estimate] \label{lem16} Let $(\boldsymbol{u},p) \in \boldsymbol{H}^2_{\tn}(\Gamma) \times H^1(\Gamma) \cap L_0^2(\Gamma)$ be the solution of (\ref{eq2}) and $(\boldsymbol{u}^e, p^e)$ be their extensions. For $(\boldsymbol{z}, w) \in \boldsymbol{H}^2_{\text{tan}}(\Gamma) \times H^1(\Gamma) \cap L_0^2(\Gamma)$ the solution of the dual problem (\ref{eq59}) and $(\boldsymbol{z}^e, w^e)$ their extensions, the following estimate holds:
    \begin{equation}
    \label{eq69}
        \mid  b_h( \tilde{\Pi}_h^{tan} \boldsymbol{z}-\boldsymbol{z}^e, p^e ) \mid \leq C h^2 \| \boldsymbol{f} \|_{\boldsymbol{L}^2(\Gamma)} \| \boldsymbol{g} \|_{\boldsymbol{L}^2(\Gamma)}.
    \end{equation}
\end{lemma}

\begin{proof}
\textcolor{black}{ Introducing $p_h$ in the first step, we bound $\text{I}_4$ by applying the Cauchy-Schwarz inequality and the result (\ref{74b}) in the next step. The last inequality is obtained by employing the interpolation estimate (\ref{I1}) and the energy error estimate (\ref{eq55}).}
    \begin{align*}
         b_h( \tilde{\Pi}_h^{tan} \boldsymbol{z}-\boldsymbol{z}^e, p^e ) & = \underbrace{b_h( \tilde{\Pi}_h^{tan} \boldsymbol{z}-\boldsymbol{z}^e, p^e- p_h )}_{\text{I}_4}+ \underbrace{ b_h( \tilde{\Pi}_h^{tan} \boldsymbol{z}-\boldsymbol{z}^e, p_h )}_{\text{I}_{5}} \\
         & \leq C \| \nabla_{\Gamma_h} (\tilde{\Pi}_h^{tan} \boldsymbol{z}-\boldsymbol{z}^e) \|_{\boldsymbol{L}^2(\Gamma_h)} \| p^e-p_h \|_{L^2(\Gamma_h)} + \text{I}_{5}\\
         & \leq C h^2 \| \boldsymbol{z}^e \|_{\boldsymbol{H}^2(\Gamma_h)} \|p^e \|_{H^1(\Gamma_h)} + \text{I}_{5}.
    \end{align*}
\textcolor{black}{To estimate the remaining term $\text{I}_{5}$, we expand it and by using the result $b_h( \tilde{\Pi}_h^{tan} \boldsymbol{z},p_h)$ = $ b(\boldsymbol{z}, p_h^l)$, applying a change of variable on the first term, and adding and subtracting $\int_{\Gamma_h} (\text{div}_{\Gamma} \boldsymbol{z})^e p_h \, ds_h$ we obtain the second equality, where we use the result (\ref{m1}) to bound the first term and the result in Lemma \ref{lemdiv} to estimate the second term. The last bound is obtained by employing the energy-error estimate (\ref{eq55}) and the regularity estimates $(\ref{eq8})$ and (\ref{eq60}).}
    \begin{align*}
         \text{I}_{5} 
          & = b(\boldsymbol{z}, p_h^l)  - b_h (\boldsymbol{z}^e, p_h ) \\
          & = \int_{\Gamma_h} (\text{div}_{\Gamma} \boldsymbol{z})^e p_h \hspace{1mm} (\mu_h -1) ds_h  + \int_{\Gamma_h} [(\text{div}_{\Gamma} \boldsymbol{z})^e - \text{div}_{\Gamma_h} \boldsymbol{z}^e] p_h \hspace{1mm} ds_h  \\
          & \leq C h^2 \| \boldsymbol{z} \|_{\boldsymbol{H}^1(\Gamma)} \| p_h \|_{L^2(\Gamma_h)} \\
          & \leq C h^2 \| \boldsymbol{z} \|_{\boldsymbol{H}^1(\Gamma)} \| p_h - p^e\|_{L^2(\Gamma_h)} + C h^2 \| \boldsymbol{z} \|_{\boldsymbol{H}^1(\Gamma)} \| p^e \|_{L^2(\Gamma_h)} \\
          & \leq C h^2 \| \boldsymbol{g} \|_{\boldsymbol{L}^2(\Gamma)} \| \boldsymbol{f} \|_{\boldsymbol{L}^2(\Gamma)}.
    \end{align*}
\end{proof}

\begin{lemma}[Primal and dual consistency error] \label{lem17}
     Let $(\boldsymbol{u}^e,p^e)$ and $(\boldsymbol{z}^e,w^e)$ be the extensions of the true solutions $(\boldsymbol{u},p) \in \boldsymbol{H}^2_{\tn}(\Gamma) \times H^1(\Gamma) \cap L_0^2(\Gamma)$ of (\ref{eq2}) and $(\boldsymbol{z},w) \in \boldsymbol{H}^2_{\tn}(\Gamma) \times H^1(\Gamma) \cap L_0^2(\Gamma)$ of the dual problem (\ref{eq59}), respectively. \textcolor{black}{For $(\boldsymbol{u}_h, p_h)$ the discrete solution of (\ref{eq13}) and $(\boldsymbol{z}_h, w_h)$ of (\ref{eq61}), the following estimates hold:}
    \begin{equation}
    \begin{split}
    \label{eq70}
        | a_h(\boldsymbol{u}_h-\boldsymbol{u}^e, \boldsymbol{z}^e)+ & b_h (\boldsymbol{u}_h-\boldsymbol{u}^e, w^e) + (\boldsymbol{u}, \boldsymbol{g})-(\boldsymbol{u}_h, \boldsymbol{g}^e)_{\Gamma_h} | \\
         & \leq C h^2 \| \boldsymbol{f} \|_{\boldsymbol{L}^2(\Gamma)} \| \boldsymbol{g} \|_{\boldsymbol{L}^2(\Gamma)},
    \end{split}
    \end{equation}
    \begin{equation}
    \begin{split}
    \label{eq71}
        | a_h(\boldsymbol{u}^e, \boldsymbol{z}_h- \boldsymbol{z}^e)+ & b_h (\boldsymbol{z}_h-\boldsymbol{z}^e, p^e) + (\boldsymbol{f}, \boldsymbol{z})-(\boldsymbol{f}^e, \boldsymbol{z}_h)_{\Gamma_h} | \\
        & \leq C h^2 \| \boldsymbol{f} \|_{\boldsymbol{L}^2(\Gamma)} \| \boldsymbol{g} \|_{\boldsymbol{L}^2(\Gamma)}.
    \end{split}
    \end{equation}
\end{lemma}
\begin{proof} It suffices to establish (\ref{eq71}), since the proof of (\ref{eq70}) proceeds in an entirely analogous manner. We decompose the error $(\boldsymbol{z}_h- \boldsymbol{z}^e)$ appearing in the bilinear forms into the auxiliary error $(\boldsymbol{z}_h- \tilde{\Pi}_h^{\text{tan}} \boldsymbol{z})$ and the interpolation error $(\tilde{\Pi}_h^{\text{tan}} \boldsymbol{z} - \boldsymbol{z}^e)$. The terms are then further rearranged by introducing $\tilde{\Pi}_h^{\text{tan}} \boldsymbol{z}$ into the expression $(\boldsymbol{f}^e, \boldsymbol{z}_h)_{\Gamma_h}$.
    \begin{align*}
        &  a_h(\boldsymbol{u}^e, \boldsymbol{z}_h- \boldsymbol{z}^e)+ b_h (\boldsymbol{z}_h-\boldsymbol{z}^e, p^e) + (\boldsymbol{f}, \boldsymbol{z})-(\boldsymbol{f}^e, \boldsymbol{z}_h)_{\Gamma_h} \\
        & = \underbrace{ \left[  a_h(\boldsymbol{u}^e, \boldsymbol{z}_h- \tilde{\Pi}_h^{\text{tan}} \boldsymbol{z}) +  b_h (\boldsymbol{z}_h- \tilde{\Pi}_h^{\text{tan}} \boldsymbol{z}, p^e) -(\boldsymbol{f}^e, \boldsymbol{z}_h-\tilde{\Pi}_h^{\text{tan}} \boldsymbol{z})_{\Gamma_h} \right]}_{\text{I}_{6}} + \\
        & \qquad \underbrace{ a_h(\boldsymbol{u}^e, \tilde{\Pi}_h^{\text{tan}} \boldsymbol{z} - \boldsymbol{z}^e) }_{\text{I}_{7}}+ \underbrace{ b_h ( \tilde{\Pi}_h^{\text{tan}} \boldsymbol{z}- \boldsymbol{z}^e, p^e)}_{\text{I}_{8}} + \underbrace{(\boldsymbol{f}, \boldsymbol{z})_{\Gamma} - (\boldsymbol{f}^e, \boldsymbol{z}^e)_{\Gamma_h}}_{\text{I}_{9}} + \\
        & \qquad  \underbrace{(\boldsymbol{f}^e,  \boldsymbol{z}^e- \tilde{\Pi}_h^{\text{tan}} \boldsymbol{z} ). }_{\text{I}_{10}}
    \end{align*}
    We estimate each of the terms separately. \textcolor{black}{To estimate $\text{I}_{6}$, in the first step we apply the consistency error estimate (\ref{eq53}) for $\boldsymbol{v}_h = (\boldsymbol{z}_h - \tilde{\Pi}_h^{\text{tan}} \boldsymbol{z})$. The last bound is obtained by subsequently applying the triangle inequality, the energy error estimate (\ref{eq62}), the interpolation estimate (\ref{I6}) and the regularity estimate (\ref{eq60}).}
\begin{align*}
    \text{I}_{6} & \leq C h \|\boldsymbol{u} \|_{\boldsymbol{H}^2(\Gamma)}  \| \boldsymbol{z}_h - \tilde{\Pi}_h^{\text{tan}} \boldsymbol{z} \|_h \\
    & \leq C h \|\boldsymbol{u} \|_{\boldsymbol{H}^2(\Gamma)} \left( \| \boldsymbol{z}_h - \boldsymbol{z}^e \|_h +  \| \boldsymbol{z}^e - \tilde{\Pi}_h^{\text{tan}} \boldsymbol{z} \|_h\right) \\
    & \leq C h^2 \|\boldsymbol{f} \|_{\boldsymbol{L}^2(\Gamma)} \|\boldsymbol{g} \|_{\boldsymbol{L}^2(\Gamma)}.
\end{align*}
\textcolor{black}{The bound for $\text{I}_7$ is derived by decomposing the symmetric tensor into its gradient and estimating the resulting terms by combining the primal-dual interpolation estimation given in Lemma 5.4 in \cite{mehlmann2023} with the fact that $\lJump \boldsymbol{u}^e \rJump$=0. Moreover, I$_{8}$ and I$_{9}$ are bounded by combining the regularity estimate (\ref{eq60}) with the Lemma \ref{lem16} and Lemma \ref{I9}, respectively.} \vspace{1mm}\\
In the end, using $\mathbf{P}^e \boldsymbol{f}^e= \boldsymbol{f}^e$, applying the Cauchy-Schwarz inequality, the interpolation estimate (\ref{I2}), the norm equivalence result (\ref{eq22}) and the regularity result (\ref{eq60}), we get
\begin{align*}
    \text{I}_{10}  = (\mathbf{P}^e \boldsymbol{f}^e, \boldsymbol{z}^e- \tilde{
    \Pi}_h^{tan} \boldsymbol{z}) 
    &  \leq \| \boldsymbol{f}^e \|_{\boldsymbol{L}^2(\Gamma_h)} \|\mathbf{P}^e(\boldsymbol{z}^e- \tilde{\Pi}_h^{tan} \boldsymbol{z}) \|_{\boldsymbol{L}^2(\Gamma_h)} \\
    & \leq C h^2 \| \boldsymbol{f} \|_{\boldsymbol{L}^2(\Gamma)} \| \boldsymbol{g} \|_{\boldsymbol{L}^2(\Gamma)}.
\end{align*}
\end{proof}

\begin{theorem}[$L^2$ Error Estimate]\label{thm4} Let $(\boldsymbol{u}, p) \in \boldsymbol{H}^2_{\tn}(\Gamma) \times H^1(\Gamma) \cap L_0^2(\Gamma)$ be the primal solution to (\ref{eq2}) and $(\boldsymbol{u}^e,p^e)$ be their extension. For $(\boldsymbol{u}_h,p_h)$ the discrete solution of (\ref{eq13}), the following error estimate holds:
    \begin{equation}
    \label{eq74}
        \| \mathbf{P}_h (\boldsymbol{u}^e - \boldsymbol{u}_h)\|_{\boldsymbol{L}^2(\Gamma_h)} \leq C h^2 \| \boldsymbol{f} \|_{\boldsymbol{L}^2(\Gamma)}.
    \end{equation}
\end{theorem}
\begin{proof}
 \begin{equation}
     \begin{split} \label{eq74a}
         \| \textbf{P}_h (\boldsymbol{u}^e - \boldsymbol{u}_h)\|_{\boldsymbol{L}^2(\Gamma_h)}^2 & = (\textbf{P}_h (\boldsymbol{u}^e - \boldsymbol{u}_h), \textbf{P}_h (\boldsymbol{u}^e - \boldsymbol{u}_h))_{\Gamma_h} \\
        & = (\boldsymbol{u}^e - \boldsymbol{u}_h, (\textbf{P}_h)^T \textbf{P}_h (\boldsymbol{u}^e - \boldsymbol{u}_h))_{\Gamma_h}\\
        & = (\boldsymbol{u}^e - \boldsymbol{u}_h,  \textbf{P}_h (\boldsymbol{u}^e - \boldsymbol{u}_h))_{\Gamma_h}\\
        & = (\boldsymbol{u}^e - \boldsymbol{u}_h, \boldsymbol{g}^e)_{\Gamma_h} \\
        & = \underbrace{(\boldsymbol{u}^e, \boldsymbol{g}^e)_{\Gamma_h}-(\boldsymbol{u}, \boldsymbol{g})_{\Gamma}}_{\text{I}_{11}} + (\boldsymbol{u}, \boldsymbol{g})_{\Gamma}-(\boldsymbol{u}_h, \boldsymbol{g}^e)_{\Gamma_h}.
     \end{split}
 \end{equation}
To treat the terms in (\ref{eq74a}) other than the geometric error in $\text{I}_{11}$, we now step-wise introduce a sequence of auxiliary terms. In this context, we first add and subtract appropriate geometric error terms and collect a set of them in $\text{I}_{12}$-$\text{I}_{14}$.
\begin{align*}
   & \| \textbf{P}_h (\boldsymbol{u}^e - \boldsymbol{u}_h)\|_{\boldsymbol{L}^2(\Gamma_h)}^2 \\
    & = \text{I}_{11} +  (\boldsymbol{u}, \boldsymbol{g})_{\Gamma}-(\boldsymbol{u}_h, \boldsymbol{g}^e)_{\Gamma_h} +\underbrace{a_h(\boldsymbol{u}^e, \boldsymbol{z}^e)- a(\boldsymbol{u}, \boldsymbol{z})}_{\text{I}_{12}}  \\
    & \quad + \underbrace{b_h(\boldsymbol{z}^e, p^e)- b(\boldsymbol{z}, p)}_{\text{I}_{13}} + \underbrace{b_h(\boldsymbol{u}^e, w^e)- b(\boldsymbol{u}, w)}_{\text{I}_{14}} \\
    & \quad + [a(\boldsymbol{u}, \boldsymbol{z})+ b(\boldsymbol{z}, p)+ b(\boldsymbol{u}, w)-a_h(\boldsymbol{u}^e, \boldsymbol{z}^e)- b_h(\boldsymbol{z}^e, p^e)- b_h(\boldsymbol{u}^e, w^e) ] .
\end{align*}
We employ the relation $ a(\boldsymbol{u}, \boldsymbol{z})+ b(\boldsymbol{z}, p)+ b(\boldsymbol{u}, w)= (\boldsymbol{f}, \boldsymbol{z})_{\Gamma} $ and further introduce the primal-dual error terms $\text{I}_{15}$-$\text{I}_{17}$ in the following.
\begin{align*}
    \| \textbf{P}_h (\boldsymbol{u}^e - \boldsymbol{u}_h)\|_{\boldsymbol{L}^2(\Gamma_h)}^2 & = \sum_{i=1}^4 \text{I}_i +  (\boldsymbol{u}, \boldsymbol{g})-(\boldsymbol{u}_h, \boldsymbol{g}^e)_{\Gamma_h}\\
    & \quad + [(\boldsymbol{f}, \boldsymbol{z})_{\Gamma}-a_h(\boldsymbol{u}^e, \boldsymbol{z}^e)- b_h(\boldsymbol{z}^e, p^e)- b_h(\boldsymbol{u}^e, w^e) ] \\
    & \quad + \underbrace{a_h(\boldsymbol{u}^e-\boldsymbol{u}_h, \boldsymbol{z}^e-\boldsymbol{z}_h)}_{\text{I}_{15}} - a_h(\boldsymbol{u}^e-\boldsymbol{u}_h, \boldsymbol{z}^e-\boldsymbol{z}_h) \\
        & \quad + \underbrace{b_h(\boldsymbol{z}^e-\boldsymbol{z}_h, p^e-p_h)}_{\text{I}_{16}} - b_h(\boldsymbol{z}^e-\boldsymbol{z}_h, p^e-p_h) \\
        & \quad + \underbrace{b_h(\boldsymbol{u}^e-\boldsymbol{u}_h, w^e-w_h)}_{\text{I}_{17}} - b_h(\boldsymbol{u}^e-\boldsymbol{u}_h, w^e-w_h)   .
\end{align*}
    We next rearrange the terms and group them so that the consistency error terms $\text{I}_{20}$ and $\text{I}_{21}$ can be identified.
    \begin{align*}
            \| \textbf{P}_h (\boldsymbol{u}^e - \boldsymbol{u}_h)\|_{\boldsymbol{L}^2(\Gamma_h)}^2 & = \sum_{i=1}^7 \text{I}_i + \underbrace{b_h(\boldsymbol{z}^e, p_h-p^e)}_{\text{I}_{18}} + \underbrace{b_h(\boldsymbol{u}^e, w_h-w^e)}_{\text{I}_{19}} \\
        & \quad + \underbrace{[a_h(\boldsymbol{u}_h - \boldsymbol{u}^e, \boldsymbol{z}^e) + b_h(\boldsymbol{u}_h - \boldsymbol{u}^e, w^e) + (\boldsymbol{u}, \boldsymbol{g})-(\boldsymbol{u}_h, \boldsymbol{g}^e)_{\Gamma_h}]}_{\text{I}_{20}} \\
        & \quad + \underbrace{[a_h(\boldsymbol{u}^e, \boldsymbol{z}_h - \boldsymbol{z}^e) + b_h(\boldsymbol{z}_h - \boldsymbol{z}^e, p^e) + (\boldsymbol{f}, \boldsymbol{z})_{\Gamma}-(\boldsymbol{f}^e, \boldsymbol{z}_h)_{\Gamma_h}],}_{\text{I}_{21}} \numberthis \label{eq61a}
    \end{align*}
    where \textcolor{black}{we use that} $(\boldsymbol{f}^e,\boldsymbol{z}_h)_{\Gamma_h}= a_h(\boldsymbol{u}_h, \boldsymbol{z}_h)+ b(\boldsymbol{z}_h, p_h)+ b(\boldsymbol{u}_h, w_h)$. \vspace{1mm}\\
The estimations of the geometric error terms $\text{I}_{11}$ and $\text{I}_{12}$ \textcolor{black}{result by combining} Lemma \ref{I9} and  Lemma \ref{lem10} (see (\ref{eq51a})), respectively, with the regularity estimates (\ref{eq8}) and (\ref{eq60}).
   Besides, the derivations of the estimates for the geometric errors of the mass conservation term $ \text{I}_{13}$ and $\text{I}_{14}$ are provided in Lemma \ref{lem11}. To estimate the primal-dual error terms $\text{I}_{15}$-$\text{I}_{17}$,\textcolor{black}{ we subsequently apply the Cauchy-Schwarz inequality, the result (\ref{74b}) together with the energy error estimates (\ref{eq55}) and (\ref{eq62}).}
    Furthermore, the estimations of the terms $\text{I}_{18}$ and $\text{I}_{19}$ are given in Lemma \ref{lem18} and the estimates for $\text{I}_{20}$ and $\text{I}_{21}$ are derived in Lemma \ref{lem17}. On combining all these results and using $\textbf{P}_h(\boldsymbol{u}^e-\boldsymbol{u}_h)$= $\boldsymbol{g}^e$, we obtain (\ref{eq74}).
\end{proof}

\section{Numerical experiment} \label{sec6}
In this section, we numerically verify the theoretical error estimates established in the previous sections and demonstrate the robustness of the surface Crouzeix-Raviart method on surfaces with different geometric characteristics. We consider two benchmark closed surfaces: the unit sphere, denoted by $\Gamma^{\text{S}}$, where
\begin{equation}
    \label{84}
    \Gamma^{\text{S}}=\{ \boldsymbol{x} = (x,y,z) \in \mathbb{R}^3: \hspace{1mm} x^2+y^2+z^2 = 1 \},
\end{equation}
and a torus with major radius 1 and minor radius 0.5, denoted by $\Gamma^{\text{T}}$, where
\begin{equation}
    \label{85}
    \Gamma^{\text{T}}= \{ \boldsymbol{x} = (x,y,z) \in \mathbb{R}^3: \hspace{1mm} \left( \sqrt{x^2+y^2} -1 \right)^2 + z^2 = \frac{1}{4}\}.
\end{equation}

\begin{figure}[htbp]
    \centering
    \begin{minipage}{0.45\textwidth}
        \centering
        \includegraphics[width=\textwidth]{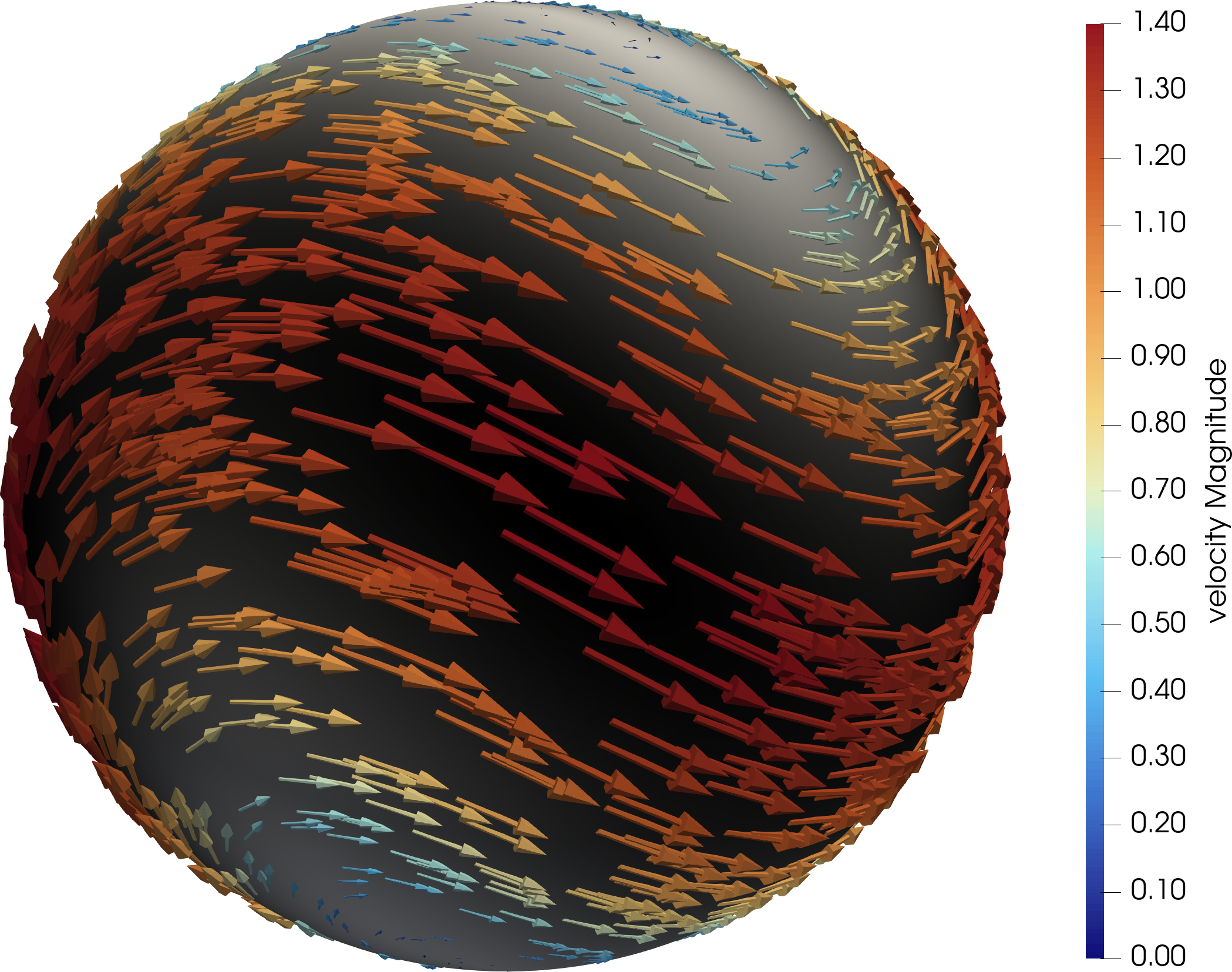}
    \end{minipage}
    \hfill 
    \begin{minipage}{0.45\textwidth}
        \centering
        \includegraphics[width=\textwidth]{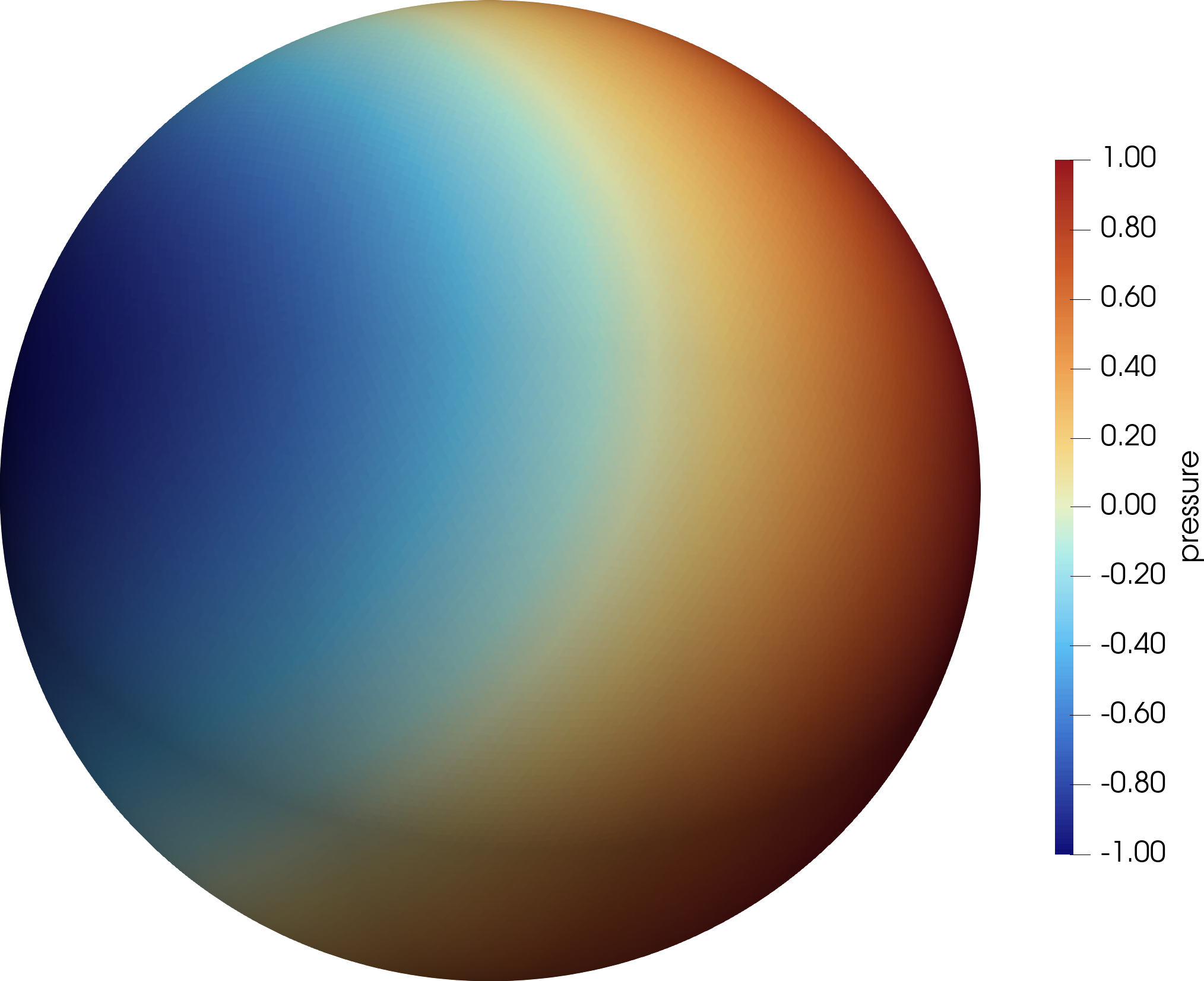}
    \end{minipage}
    \caption{Computed velocity (left) and pressure (right) solutions of the surface Stokes problem (\ref{eq1}) on the unit sphere $\Gamma^{\text{S}}$ to the analytic solution (\ref{86}) and (\ref{87}), respectively.}
    \label{fig2}
\end{figure}

\begin{table}[h]\centering
\caption{Experimental $\boldsymbol{H}^1$-errors and order of convergences on the sphere for different stabilization parameters $\alpha$ (\ref{eq13}). The error in the velocity field is denoted by $e_{\boldsymbol{u}}$, where the analytical solution is given in (\ref{86}).}
  \label{tab1}
\begin{tabular}{lcccccc}
\toprule
\multirow{2}[3]{*}{dof} & \multicolumn{2}{c}{$\alpha=0.5$} & \multicolumn{2}{c}{$\alpha=1.0$} & \multicolumn{2}{c}{$\alpha=2.0$} \\
\cmidrule(lr){2-3} \cmidrule(lr){4-5} \cmidrule(lr){6-7}
 & $ | e_{\boldsymbol{u}}|_{\boldsymbol{H}^1(\Gamma_h)}$ & Order & $|e_{\boldsymbol{u}}|_{\boldsymbol{H}^1(\Gamma_h)}$ & Order & $| e_{\boldsymbol{u}}|_{\boldsymbol{H}^1(\Gamma_h)}$ & Order \\
\midrule
960 & 2.52$e^{-01}$ & 0.00 & 2.60$e^{-01}$ & 0.00 & 2.83$e^{-01}$ & 0.00 \\
3840 & 8.52$e^{-02}$ & 1.57 & 9.09$e^{-02}$ & 1.52 & 1.08$e^{-01}$ & 1.38 \\
15360 & 3.51$e^{-02}$ & 1.28 & 3.82$e^{-02}$ & 1.25 & 4.85$e^{-02}$ & 1.16 \\
61440 & 1.65$e^{-02}$ & 1.09 & 1.80$e^{-02}$ & 1.08 & 2.35$e^{-02}$ & 1.05 \\
245760 & 8.08$e^{-03}$ & 1.03 & 8.86$e^{-03}$ & 1.02 & 1.16 $e^{-02}$ & 1.01 \\
\bottomrule
\end{tabular}
\end{table}

   \begin{table}[h]\centering
\caption{Experimental $\boldsymbol{L}^2$-errors and order of convergences on the sphere for different stabilization parameters $\alpha$ (\ref{eq13}). The projected $\boldsymbol{L}^2$-error of the velocity field is denoted by $\mathbf{P}_h e_{\boldsymbol{u}}$, where the analytic solution is given in (\ref{86}).}
  \label{tab2}
\begin{tabular}{lcccccc}
\toprule
\multirow{2}[3]{*}{Dof} & \multicolumn{2}{c}{$\alpha=0.5$} & \multicolumn{2}{c}{$\alpha=1.0$} & \multicolumn{2}{c}{$\alpha=2.0$} \\
\cmidrule(lr){2-3} \cmidrule(lr){4-5} \cmidrule(lr){6-7}
 & $\| \textbf{P}_h e_{\boldsymbol{u}}\|_{\boldsymbol{L}^2(\Gamma_h)}$ & Order & $\|  \textbf{P}_h e_{\boldsymbol{u}}\|_{\boldsymbol{L}^2(\Gamma_h)}$ & Order & $\| \textbf{P}_h e_{\boldsymbol{u}}\|_{\boldsymbol{L}^2(\Gamma_h)}$ & Order \\
\midrule
960 & 1.28$e^{-01}$ & 0.00 & 1.32$e^{-01}$ & 0.00 & 1.39$e^{-01}$ & 0.00 \\
3840 & 3.33$e^{-02}$ & 1.94 & 3.44$e^{-02}$ & 1.94 & 3.63$e^{-02}$ & 1.94 \\
15360 & 8.41$e^{-03}$ & 1.98 & 8.68$e^{-03}$ & 1.98 & 9.16$e^{-03}$ & 1.99 \\
61440 & 2.11$e^{-03}$ & 2.00 & 2.18$e^{-02}$ & 2.00 & 2.30$e^{-03}$ & 2.00 \\
245760 & 5.28$e^{-04}$ & 2.00 & 5.45$e^{-04}$ & 2.00 & 5.75 $e^{-04}$ & 2.00 \\
\bottomrule
\end{tabular}
\end{table}

 \begin{table}[h]\centering
\caption{Experimental $L^2$-errors and order of convergences on the sphere for different stabilization parameters $\alpha$ (\ref{eq13}). The error in the pressure is denoted by $ e_{p}$, where the analytic solution is given in (\ref{87}).}
  \label{tab3}
\begin{tabular}{lcccccc}
\toprule
\multirow{2}[3]{*}{Dof} & \multicolumn{2}{c}{$\alpha=0.5$} & \multicolumn{2}{c}{$\alpha=1.0$} & \multicolumn{2}{c}{$\alpha=2.0$} \\
\cmidrule(lr){2-3} \cmidrule(lr){4-5} \cmidrule(lr){6-7}
 & $\|  e_p\|_{L^2(\Gamma_h)}$ & Order & $\| e_p \|_{L^2(\Gamma_h)}$ & Order & $\| e_p\|_{L^2(\Gamma_h)}$ & Order \\
\midrule
320 & 4.83$e^{-02}$ & 0.00 & 4.85$e^{-01}$ & 0.00 & 4.90$e^{-02}$ & 0.00 \\
1280 & 1.53$e^{-02}$ & 1.65 & 1.40$e^{-02}$ & 1.79 & 1.36$e^{-02}$ & 1.85 \\
5120 & 4.88$e^{-03}$ & 1.65 & 3.98$e^{-03}$ & 1.82 & 3.65$e^{-03}$ & 1.90 \\
20480 & 1.63$e^{-03}$ & 1.58 & 1.22$e^{-03}$ & 1.71 & 1.04$e^{-03}$ & 1.80 \\
81920 & 6.32$e^{-04}$ & 1.37 & 4.48$e^{-04}$ & 1.45 & 3.62 $e^{-04}$ & 1.53 \\
\bottomrule
\end{tabular}
\end{table}

To avoid cumbersome notation, we henceforth use $\Gamma$ to denote the underlying surface, where $\Gamma$ represents either $\Gamma^{\text{S}}$ or $\Gamma^{\text{T}}$, depending on the example under consideration. The method is implemented in Python, following the discretization strategy employed in the climate model ICON \cite{MehlmannGutjahr2022}. We generate the corresponding approximate surfaces $\Gamma_h$ using the Python package $PyVista$ \cite{sullivan19}. For evaluating the errors, we prescribe an exact solution $(\boldsymbol{u},p)$ to (\ref{eq1}) and calculate the corresponding given data $\boldsymbol{f}$ from (\ref{eq1}) using $SymPy$ \cite{Meurer17}. The exact velocity is defined in terms of a stream function, $\boldsymbol{\psi}(x,y,z)$= $(xy+z)$ as follows
\begin{equation}
\label{86}
    \boldsymbol{u}= \boldsymbol{n} \times \nabla_{\Gamma} \boldsymbol{\psi},
\end{equation}
where the unit outward normal vector $\boldsymbol{n}$ can be determined analytically from the level-set formulation of the corresponding surface and the exact pressure solution is 
\begin{equation}
    \label{87}
    p=(x+yz).
\end{equation}
In this section, we denote the velocity error, $(\boldsymbol{u}^e-\boldsymbol{u}_h)$, and pressure error, $(p^e-p_h)$, by $e_{\boldsymbol{u}}$ and $e_p$, respectively. Here, we assess the accuracy of the method by measuring the velocity error in the broken $\boldsymbol{H}^1$ semi-norm and the broken $\boldsymbol{L}^2$-norm, defined in Section \ref{sec4.1} and the pressure error with respect to the $L^2$-norm, evaluated on the respective approximate surfaces. To investigate the influence of the stabilization parameter (\ref{eq13}), we perform numerical experiments on both surfaces for three choices of $\alpha$, namely $\alpha$ = 0.5, 1.0 and 2.0.

\begin{figure}[htbp]
    \centering
    \begin{minipage}{0.49\textwidth}
        \centering
        \includegraphics[width=\textwidth]{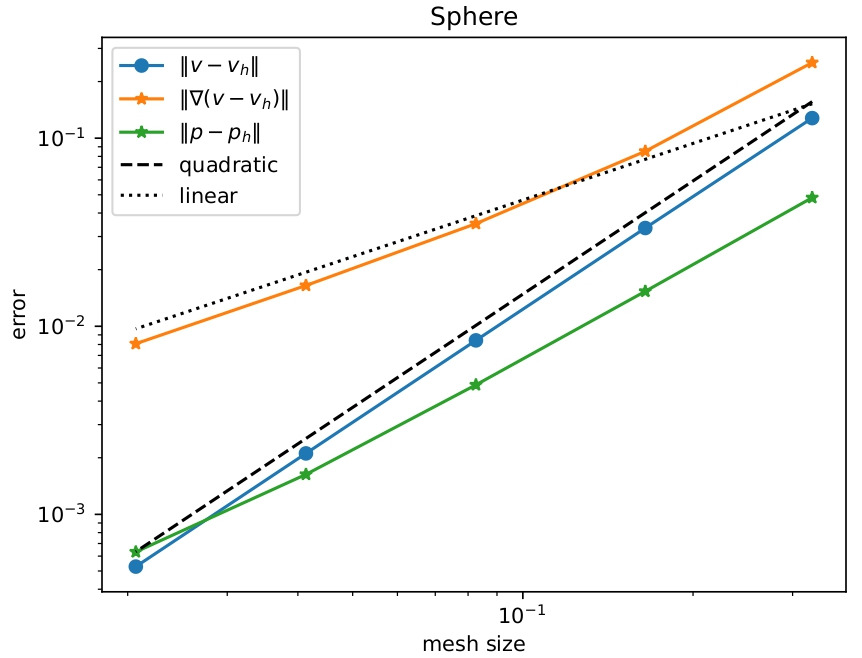}
    \end{minipage}
    \hfill 
    \begin{minipage}{0.49\textwidth}
        \centering
        \includegraphics[width=\textwidth]{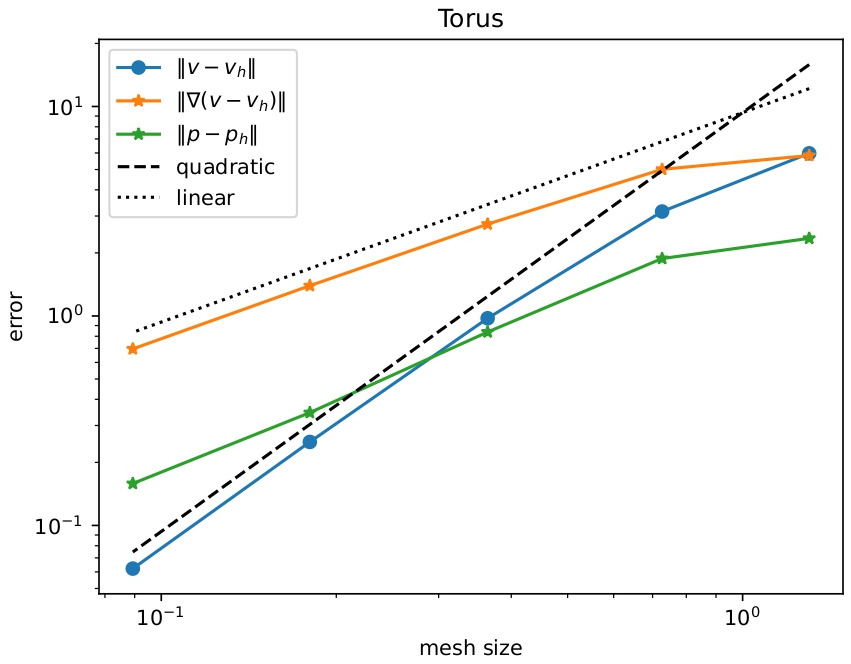}
    \end{minipage}

    \begin{minipage}{0.49\textwidth}
        \centering
        \includegraphics[width=\textwidth]{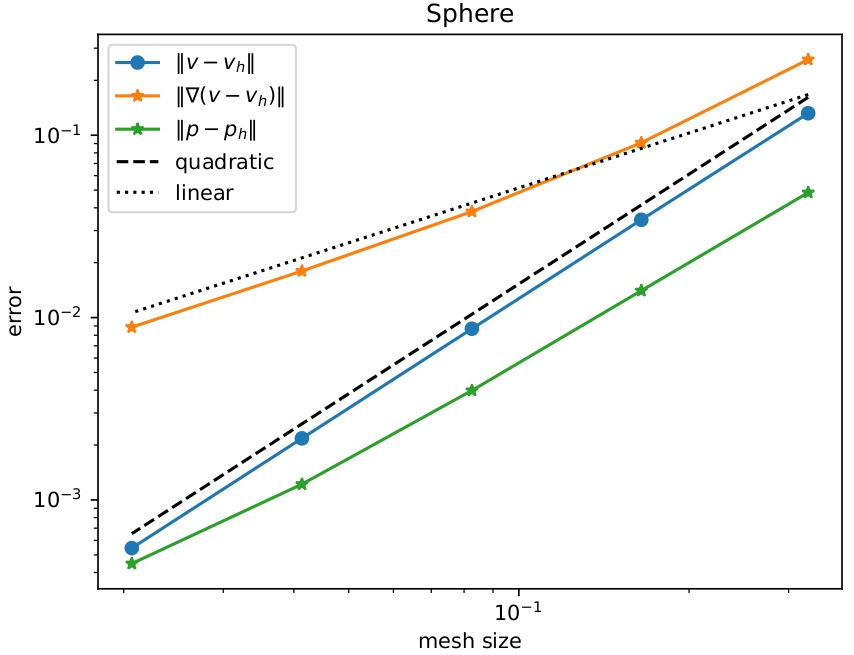}
    \end{minipage}
    \hfill 
    \begin{minipage}{0.49\textwidth}
        \centering
        \includegraphics[width=\textwidth]{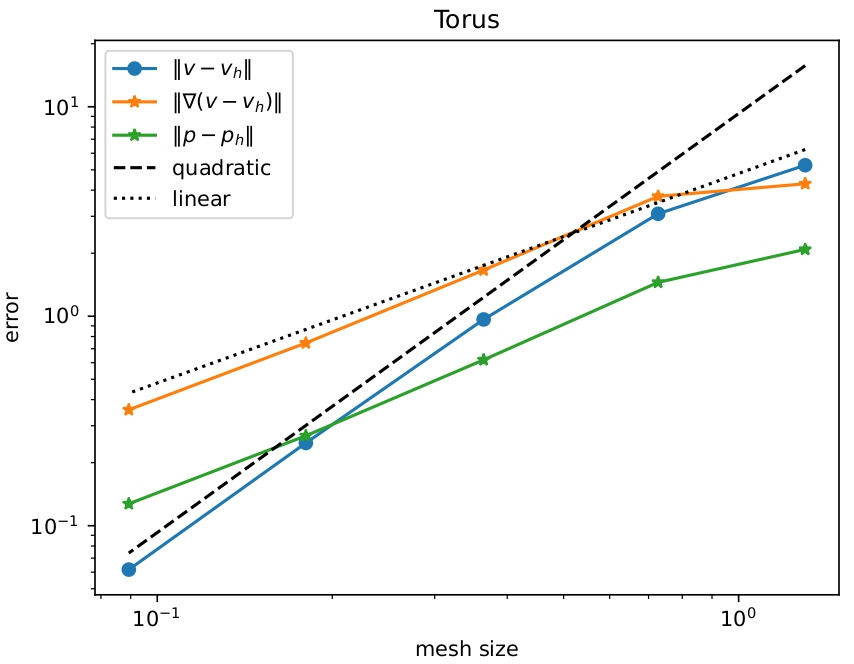}
    \end{minipage}

    \begin{minipage}{0.49\textwidth}
        \centering
        \includegraphics[width=\textwidth]{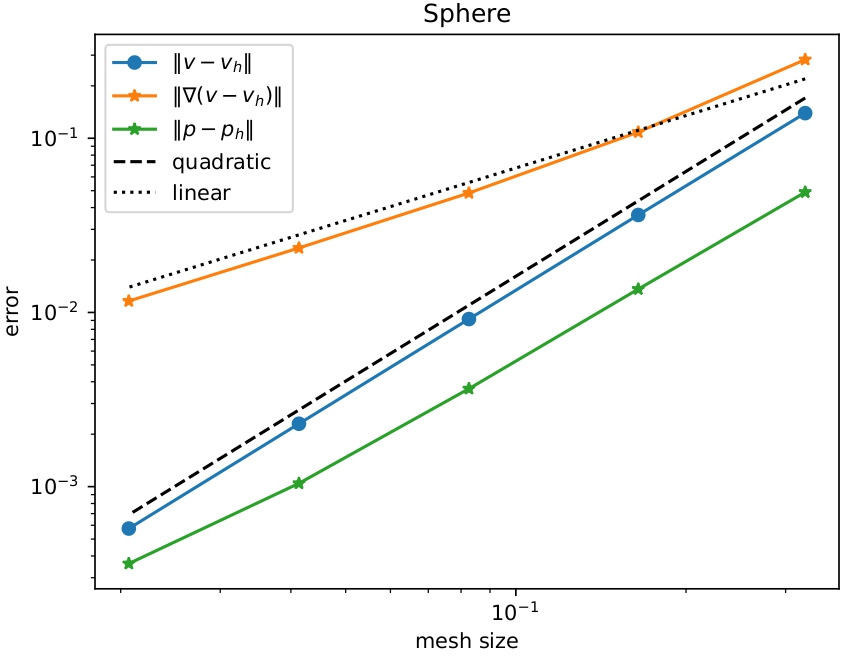}
    \end{minipage}
    \hfill 
    \begin{minipage}{0.49\textwidth}
        \centering
        \includegraphics[width=\textwidth]{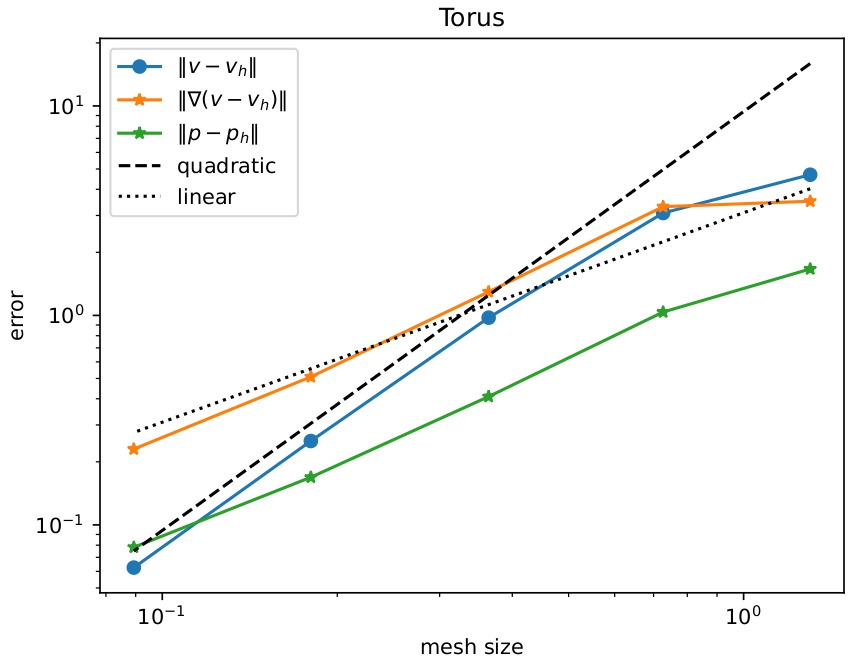}
    \end{minipage}
\caption{ $\boldsymbol{H}^1$-error velocity field (orange), projected $\boldsymbol{L}^2$-error of the velocity field (blue) and $L^2$-error in the pressure evaluated on the sphere (left) and the torus (right). The errors are calculated with respect to the analytic solution given in (\ref{86}) and (\ref{87}), respectively. The dotted lines indicate the first order convergence, whereas the dashed lines display the second order convergence.  }
    \label{fig8}
\end{figure} 

\subsection{Results on the Sphere}
Figure \ref{fig2} depicts the velocity and pressure solutions obtained using the finite element formulation (\ref{eq13}) of the surface Stokes problem (\ref{eq1}) on the unit sphere (\ref{84}) with the stabilization parameter, $\alpha$=1.0. For this setup, Table \ref{tab1} presents the $\boldsymbol{H}^1$ semi-norm of the velocity error $e_{\boldsymbol{u}}$, evaluated for a sequence of five successively refined meshes and for three values of $\alpha$=0.5, 1.0 and 2.0. In all cases, the proposed method exhibits the predicted first-order convergence rate, obtained in Theorem \ref{thm3}. The smallest errors are consistently observed for $\alpha=0.5$. Furthermore, Table \ref{tab2} and Table \ref{tab3} present the projected velocity error, $\textbf{P}_h(\boldsymbol{u}^e- \boldsymbol{u}_h)$ and the pressure error with respect to the $L^2$-norm. The numerical results demonstrate second-order convergence for the projected velocity and first-order convergence for the pressure, in agreement with the theoretical error estimates, established in Theorem \ref{thm4}. Compared to the $\boldsymbol{H}^1$-errors of the velocity field, we see that the $\boldsymbol{L}^2$-velocity errors are relatively insensitive to the choice of the stabilization parameter. As shown in Table \ref{tab3}, we observe that the $L^2$- error in the pressure field decreases as $\alpha$ increases. Figure \ref{fig8} illustrates the convergence behavior of the velocity and pressure errors as functions of the mesh size, providing a clear visual confirmation of the observed rates in both the $\boldsymbol{H}^1$ semi-norm and the $\boldsymbol{L}^2$-norm.

\begin{figure}[htbp]
    \begin{minipage}{0.45\textwidth}
        \centering
        \includegraphics[width=\textwidth]{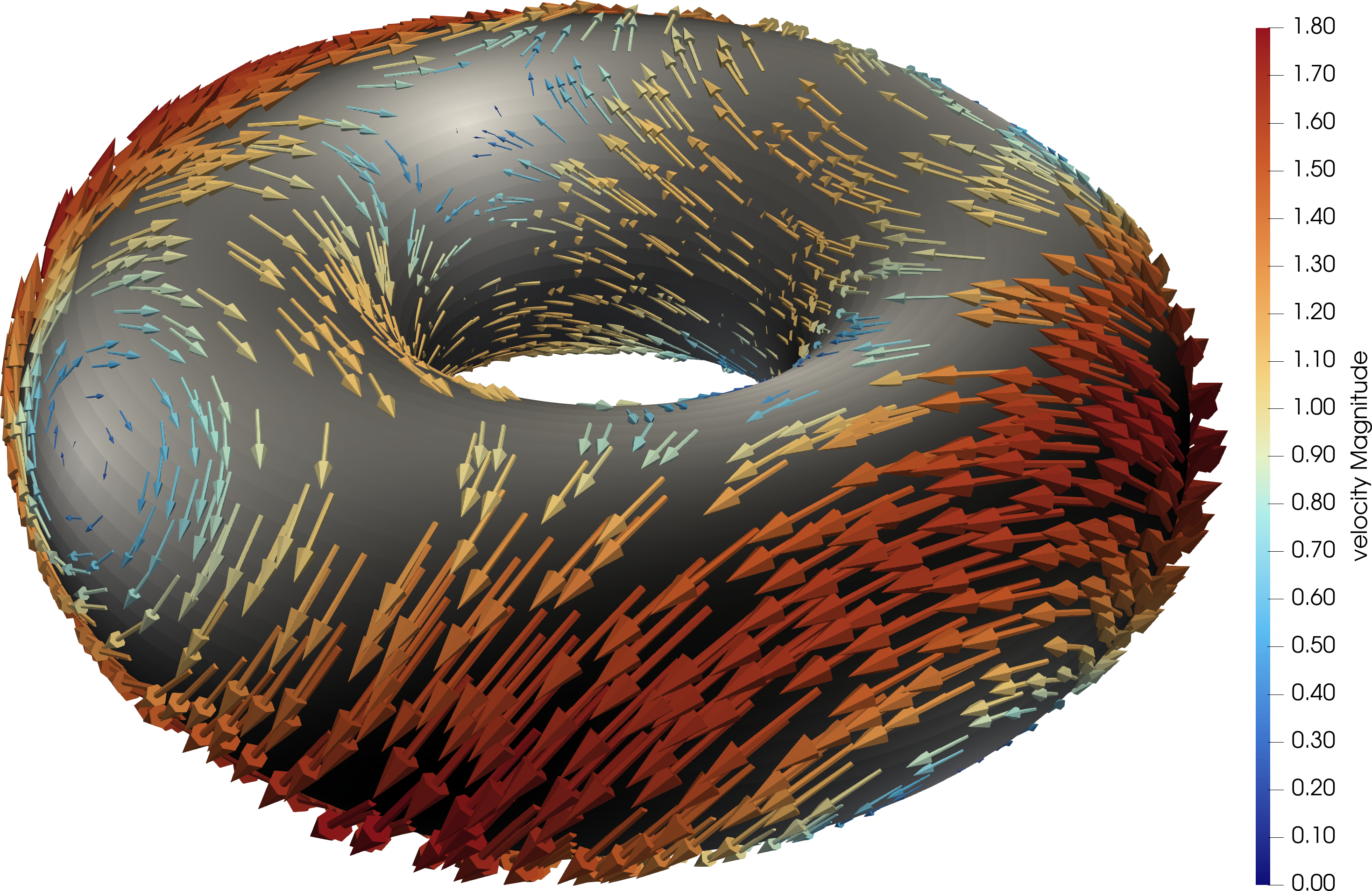}
    \end{minipage}
    \hfill 
    \begin{minipage}{0.45\textwidth}
        \centering
        \includegraphics[width=\textwidth]{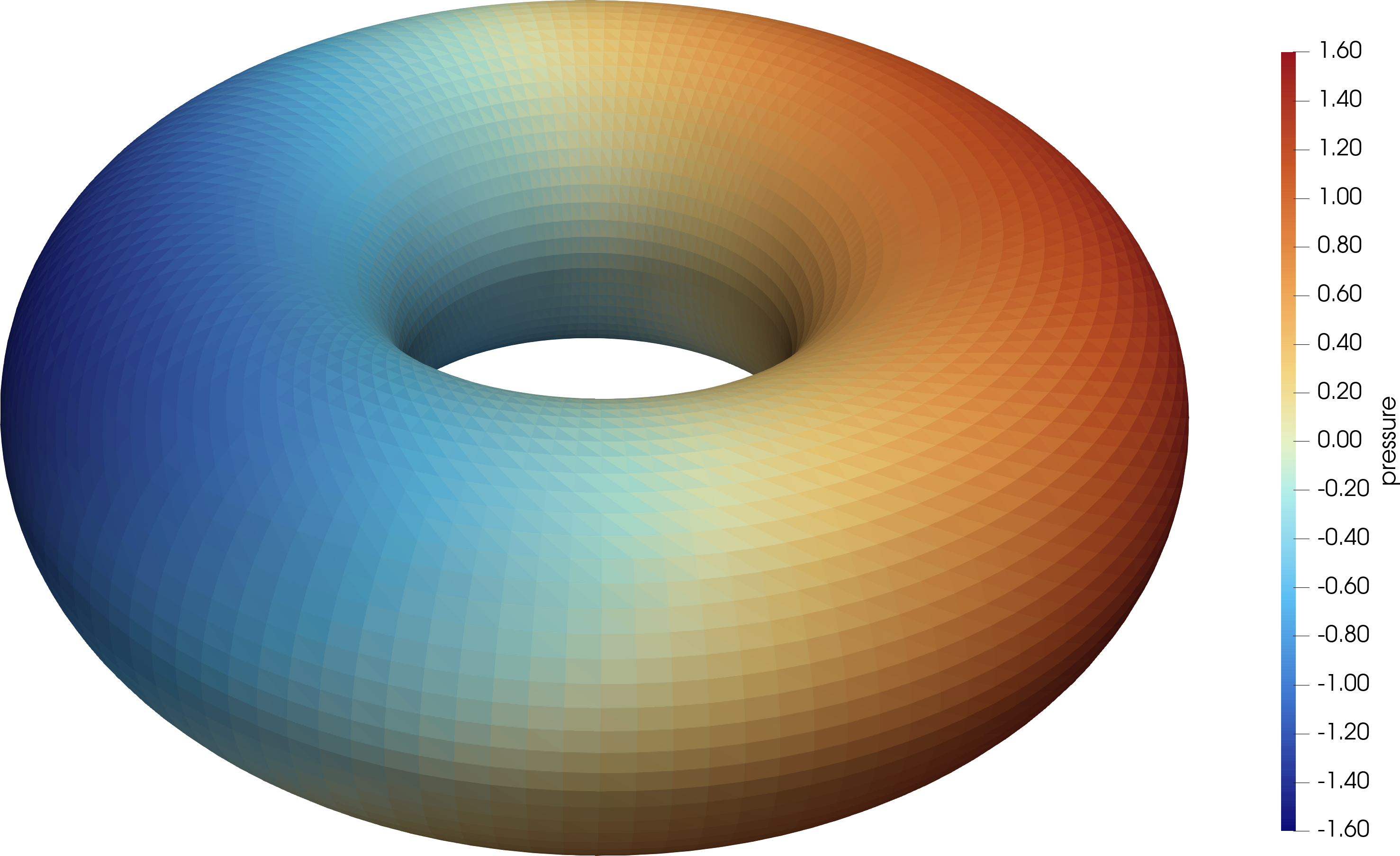}
    \end{minipage}
    \caption{Computed velocity (left) and pressure (right) solutions of the surface Stokes problem (\ref{eq1}) on the unit torus $\Gamma^{\text{T}}$ to the analytic solution (\ref{86}) and (\ref{87}), respectively.}
    \label{fig2a}
\end{figure}

\subsection{Results on the Torus}
Figure \ref{fig2a} presents the computed finite element velocity and pressure solutions  of the surface Stokes problem (\ref{eq1}) on the torus for the stabilization parameter, $\alpha=1.0$. The corresponding errors and experimental convergence rates for $\alpha=0.5$, $\alpha=1.0$ and $\alpha=2.0$ are reported in Tables \ref{tab4} - \ref{tab6}. The corresponding convergence plots are presented in Figure \ref{fig8}, illustrating the decay of velocity and pressure errors as the mesh size decreases. As in the case of the unit sphere, the numerical results exhibit the convergence rates predicted by the theoretical error estimates established in Theorem \ref{thm3} and Theorem \ref{thm4} for all choices of the stabilization parameter. In particular, the velocity error in the $\boldsymbol{H}^1$ semi-norm converges with first order, while the projected velocity error measured in the $\boldsymbol{L}^2$-norm exhibits second-order convergence. The pressure error also converges with the expected first-order rate in the $L^2$-norm. We observe that for each discretization level, the velocity error measured in the $\boldsymbol{H}^1$ semi-norm and the pressure error measured in the $L^2$-norm decrease as $\alpha$ increases and smallest errors are obtained for $\alpha=2.0$. In contrast, compared to the velocity error measured in the $\boldsymbol{H}^1$-norm, the projected velocity error in the $\boldsymbol{L}^2$-norm remains largely insensitive to variations in the stabilization parameter $\alpha$.

\begin{table}[htbp]
\caption{Experimental $\boldsymbol{H}^1$-errors and order of convergences on the torus for different stabilization parameters $\alpha$ (\ref{eq13}). The error in the velocity field is denoted by $e_{\boldsymbol{u}}$, where the analytical solution is given in (\ref{86}).}
  \label{tab4}
\begin{tabular}{lcccccc}
\toprule
\multirow{2}[3]{*}{Dof} & \multicolumn{2}{c}{$\alpha=0.5$} & \multicolumn{2}{c}{$\alpha=1.0$} & \multicolumn{2}{c}{$\alpha=2.0$} \\
\cmidrule(lr){2-3} \cmidrule(lr){4-5} \cmidrule(lr){6-7}
 & $ | e_{\boldsymbol{u}}|_{\boldsymbol{H}^1(\Gamma_h)}$ & Order & $|e_{\boldsymbol{u}}|_{\boldsymbol{H}^1(\Gamma_h)}$ & Order & $| e_{\boldsymbol{u}}|_{\boldsymbol{H}^1(\Gamma_h)}$ & Order \\
\midrule
126 & 5.83$e^{+00}$ & 0.00 & 4.29$e^{+00}$ & 0.00 & 3.51$e^{+00}$ & 0.00 \\
630 & 5.01$e^{+00}$ & 0.22 & 3.74$e^{+00}$ & 0.20 & 3.30$e^{+00}$ & 0.09 \\
2790 & 2.74$e^{+00}$ & 0.87 & 1.66$e^{+00}$ & 1.17 & 1.30$e^{+00}$ & 1.35 \\
11718 & 1.39$e^{+00}$ & 0.98 & 7.45$e^{-01}$ & 1.16 & 5.09$e^{-01}$ & 1.35 \\
48006 & 6.96$e^{-01}$ & 1.00 & 3.58$e^{-01}$ & 1.06 & 2.30$e^{-01}$ & 1.15 \\
\bottomrule
\end{tabular}
\end{table}

 \begin{table}[htbp]\centering
\caption{Experimental $\boldsymbol{L}^2$-errors and order of convergences on the torus for different stabilization parameters $\alpha$ (\ref{eq13}). The projected $\boldsymbol{L}^2$-error of the velocity field is denoted by $\mathbf{P}_h e_{\boldsymbol{u}}$, where the analytic solution is given in (\ref{86}).}
  \label{tab5}
\begin{tabular}{lcccccc}
\toprule
\multirow{2}[3]{*}{Dof} & \multicolumn{2}{c}{$\alpha=0.5$} & \multicolumn{2}{c}{$\alpha=1.0$} & \multicolumn{2}{c}{$\alpha=2.0$} \\
\cmidrule(lr){2-3} \cmidrule(lr){4-5} \cmidrule(lr){6-7}
 & $\| \textbf{P}_h e_{\boldsymbol{u}}\|_{\boldsymbol{L}^2(\Gamma_h)}$ & Order & $\|  \textbf{P}_h e_{\boldsymbol{u}}\|_{\boldsymbol{L}^2(\Gamma_h)}$ & Order & $\| \textbf{P}_h e_{\boldsymbol{u}}\|_{\boldsymbol{L}^2(\Gamma_h)}$ & Order \\
\midrule
126 & 5.98$e^{+00}$ & 0.00 & 5.26$e^{+00}$ & 0.00 & 4.70$e^{+00}$ & 0.00 \\
630 & 3.15$e^{+00}$ & 0.92 & 3.08$e^{+00}$ & 0.77 & 3.08$e^{+00}$ & 0.61 \\
2790 & 9.75$e^{-01}$ & 1.69 & 9.65$e^{-01}$ & 1.68 & 9.75$e^{-01}$ & 1.66 \\
11718 & 2.50$e^{-01}$ & 1.96 & 2.48$e^{-01}$ & 1.96 & 2.51$e^{-01}$ & 1.96 \\
48006 & 6.22$e^{-02}$ & 2.01 & 6.17$e^{-02}$ & 2.01 & 6.25$e^{-02}$ & 2.01 \\
\bottomrule
\end{tabular}
\end{table}

 \begin{table}[htbp]\centering
\caption{Experimental $L^2$-errors and order of convergences on the torus for different stabilization parameters $\alpha$ (\ref{eq13}). The error in the pressure is denoted by $ e_{p}$, where the analytic solution is given in (\ref{87}).}
  \label{tab6}
\begin{tabular}{lcccccc}
\toprule
\multirow{2}[3]{*}{Dof} & \multicolumn{2}{c}{$\alpha=0.5$} & \multicolumn{2}{c}{$\alpha=1.0$} & \multicolumn{2}{c}{$\alpha=2.0$} \\
\cmidrule(lr){2-3} \cmidrule(lr){4-5} \cmidrule(lr){6-7}
 & $\|  e_p\|_{L^2(\Gamma_h)}$ & Order & $\| e_p \|_{L^2(\Gamma_h)}$ & Order & $\| e_p\|_{L^2(\Gamma_h)}$ & Order \\
\midrule
42 & 2.35$e^{+00}$ & 0.00 & 2.80$e^{+00}$ & 0.00 & 1.67$e^{+00}$ & 0.00 \\
210 & 1.88$e^{+00}$ & 0.32 & 1.45$e^{+00}$ & 0.52 & 1.04$e^{+00}$ & 0.68 \\
930 & 8.38$e^{-01}$ & 1.17 & 6.19$e^{-01}$ & 1.20 & 4.10$e^{-01}$ & 1.34 \\
3906 & 3.46$e^{-01}$ & 1.28 & 2.69$e^{-01}$ & 1.20 & 1.69$e^{-01}$ & 1.28 \\
16002 & 1.58$e^{-01}$ & 1.12 & 1.27$e^{-01}$ & 1.08 & 7.82 $e^{-02}$ & 1.11 \\
\bottomrule
\end{tabular}
\end{table}


\end{document}